\documentclass[12pt]{article}
\usepackage{amssymb}
\usepackage{amsthm}
\usepackage[a4paper, total={16cm, 25cm}]{geometry}

\usepackage{amsmath}	
\usepackage[noblocks]{authblk}

\newtheorem{assumption}{Assumption}
\newtheorem{theorem}{Theorem}
\newtheorem{remark}{Remark}
\usepackage{xcolor}
\usepackage{tcolorbox}
\usepackage{bm}
\usepackage{subcaption}
\usepackage{algorithm}
\usepackage[title,titletoc]{appendix}
\usepackage[noend]{algpseudocode}
\def\BState{\State\hskip-\ALG@thistlm}

 \usepackage{todonotes}
\usepackage[hidelinks]{hyperref}
\usepackage{cleveref}

\usepackage{booktabs,dcolumn}

\newcolumntype{d}[1]{D{.}{.}{#1}} 

\def \d  {\delta}

\def \eps {\epsilon}

\def \k  {\kappa}

\def \L  {\mathfrak{L}}

\def \L  {\mathcal{L}}

\def \Om {\Omega}

\def \calT {\mathcal{T}}
\def \calP {\mathcal{P}}

\def \t  {\tau}

\def \del {\nabla}

\def \p  {\partial}
\def \R  {\mathbb{R}}
\def \N  {\mathbb{N}}

\def \LN {\!_{L\to N}}
\def \NL {\!_{N\to N}}
\def \NL {\!_{N\to L}}
\def \LL {\!_{L\to L}}
\def \efl {\bm{\sigma}^i_L}
\def \efn {\bm{\sigma}^i_N}

\newcommand{\norm}[1]{{\left\vert\kern-0.25ex\left\vert\kern-0.25ex\left\vert #1 
    \right\vert\kern-0.25ex\right\vert\kern-0.25ex\right\vert}}

\usepackage{pgfplots}
\usetikzlibrary{shapes.geometric, arrows}
\usepackage{xcolor}
\definecolor{colorL2}{RGB}{255 153 51}
\definecolor{colorL1}{RGB}{255 0 255}
\definecolor{colorLN}{RGB}{255 40 40}
\definecolor{colorNewton}{RGB}{71 150 255}

\newtheorem{proposition}[theorem]{Proposition}

\newtheorem{definition}{Definition}[section]

\title{An adaptive solution strategy for Richards' equation}

\author[1]{Jakob S. Stokke}

\author[2]{Koondanibha Mitra}

\author[1,3]{Erlend Storvik}

\author[1]{Jakub W. Both}

\author[1]{Florin A. Radu}

\affil[1]{Center for Modeling of Coupled Subsurface Dynamics, Department of Mathematics, University of Bergen, Norway}

\affil[2]{Computational Mathematics group, Hasselt University, Belgium}

\affil[3]{Department of Computer science, Electrical engineering and Mathematical sciences, Western Norway University of Applied Sciences, Norway}

\providecommand{\keywords}[1]{\textbf{Keywords:} \textit{#1}}

\begin{document}
\maketitle

\begin{abstract}
Flow in variably saturated porous media is typically modelled by the Richards equation, a nonlinear elliptic-parabolic equation which is notoriously challenging to solve numerically. In this paper, we propose a robust and fast iterative solver for Richards' equation. The solver relies on an adaptive switching algorithm, based on rigorously derived {\it a posteriori} indicators, between two linearization methods: L-scheme and Newton. Although a combined L-scheme/Newton strategy was introduced previously in \cite{ListFlorian2016Asoi}, here, for the first time we propose a reliable and robust criteria for switching between these schemes. The performance of the solver, which can be in principle applied to any spatial discretization and linearization methods, is illustrated through several numerical examples. 
\end{abstract}





\section{Introduction}
\label{sec: intro}

In this paper, we consider the pressure head $\psi$ based formulation of the Richards equation
\begin{equation}\label{eq: Richards equation}
	\partial_t\theta(\psi)-\nabla\cdot\left[K(\theta(\psi))\nabla(\psi+z))\right]=f,
\end{equation}
where $\theta:\R\to [0,1]$ is the water content, $K$ is the rank 2 permeability tensor of the porous medium, $z$ is the height against the gravitational direction, and $f$ is a source/sink term. Richards' equation is used to model the flow of water in saturated/unsaturated porous media. It is a highly nonlinear and degenerate elliptic-parabolic equation which makes solving it a very challenging task, see e.g. the review work of \cite{farthing}. We refer to \cite{alt1983} for the existence and uniqueness of a weak solution of Richards' equation. 




There are plenty of works regarding discretization of Richards' equation. Due to the low regularity of solutions of \eqref{eq: Richards equation}, see \cite{alt1984nonstationary}, generally, a backward Euler (implicit) scheme \eqref{eq: time-discrete Richards equation} is  employed to discretize it in time, see e.g. \cite{ListFlorian2016Asoi,radu2008}. Regarding spatial discretization we mention continuous Galerkin finite elements \cite{Arbogast1990,Arbogast1993}, mixed or expanded mixed finite elements \cite{ArbogastWheelerZhang, Woodward, radu2006convergence,Radu2014, Bause}, finite volumes \cite{EymardR1999Tfvm,EymardRobert2006Acfv} (see also the recent review \cite{bassseto2022}), or  multipoint flux approximation (MPFA) \cite{klausen}. Regardless of the choice of the spatial discretization method, one has to solve at each time step a nonlinear, finite-dimensional problem. In this paper, we will focus on how to efficiently solve these problems using iterative linearization techniques.


 The main iterative linearization methods used for this type of nonlinear problem are the Newton method, Picard or modified Picard, L-scheme, the Jaeger-Kacur method, or combinations of them. Perhaps the most common choice is the Newton method \cite{bergamaschi1999mixed,lehmann1998comparison} which converges quadratically provided the initial guess is close enough to the final solution. For a $r$-H\"older continuous  $\theta'$ function ($r\in (0,1]$) and the initial guess equal to the solution of the previous time step, it was shown in \cite{radu2006convergence} that the Newton scheme is $(1+r)^{\rm th}$ order convergent if
\begin{align}\label{eq: convergence cond Newton}
    \t\leq  C \theta_m^{\frac{2+r}{r}} h^d,
\end{align}
where $\t>0$ is the time step size, $h>0$ the mesh size, $d\in \N$ the spatial dimension, $C>0$ a constant which depends on the domain and the nonlinearities, and $\theta_m:=\inf \theta'\geq 0$. However, for simulations in 2 or 3 dimensions, condition~\eqref{eq: convergence cond Newton} is quite restrictive particularly if the mesh size $h$ is small, or if the problem is degenerate ($\theta_m=0$). This fact is corroborated by numerical simulations in \cite{ListFlorian2016Asoi,MitraK.2019AmLt} which show that the Newton method fails to converge in many such cases. One can improve the robustness of Newton method by using a damped version of it. Line search, variable switching \cite{brenner2017improving} or trust-regions techniques \cite{hamdi} are examples of such. Alternatively, one can increase the robustness of Newton's method by performing first a few fixed-point iterations. This was proposed in \cite{bergamaschi1999mixed,lehmann1998comparison} by using the Picard method and in \cite{ListFlorian2016Asoi} by using the L-scheme. Nevertheless, the switching between the schemes was not based on an {\it{a posteriori}} indicator, but done in a heuristic manner. 

The other linearization schemes are fixed-point type schemes, typically more robust, however only linearly convergent. It has been shown in \cite{celia_Picard,EymardR1999Tfvm} that the Picard method does not perform well for Richards' equation. A modified Picard method was proposed in \cite{celia_Picard}. The modified Picard coincides with Newton's method for the case of a constant permeability, therefore it inherits robustness problems. The L-scheme, first proposed in \cite{PopI.S.2004Mfef,SlodickaM2002Arae,ListFlorian2016Asoi}, is a stabilized Picard method and it was designed to be unconditionally converging irrespective of the choice of the initial guess even in degenerate settings and for larger time steps. The L-scheme (see \Cref{def: L-scheme}) uses a global constant as a  stabilization coefficient, does not involve the computation of any derivatives, and thus, is not only more stable but also consumes less computational time per iteration due to easier assembly of the  stiffness matrices which are better conditioned. Numerical results in \cite{ListFlorian2016Asoi,MitraK.2019AmLt} clearly demonstrate this. However, they also reveal that the L-scheme converges considerably slower in terms of number of iterations compared to the Newton scheme and at a linear rate. Furthermore, its overall performance strongly depends on the careful choice of a tuning parameter; despite theoretical stability, an improper choice may effectively result in stagnation. The sensitivity of the performance of the L-scheme with respect to the stabilization can be significantly relaxed when combining the L-scheme with Anderson acceleration~\cite{anderson1965iterative}. Indeed, for Richards equation extended to deformable porous media and solved by an L-scheme, it has been demonstrated that, first, the stabilization parameter can be chosen outside the theoretical range, and second, the non-degenerate convergence can be retained in case of previous divergence or accelerated, as also discussed from a theoretical perspective~\cite{both2019anderson}. Similar stabilizing properties of the Anderson acceleration have been also discussed for general fixed-point methods~\cite{evans2020proof, pollock2021anderson}. Other fixed point iterations schemes include J\"ager-Kac\v{u}r scheme \cite{jager1995solution} which converges unconditionally albeit slowly, and is more computationally expensive than the L-scheme per iteration, see \Cref{tab:linear schemes}. The modified L-scheme, proposed in \cite{MitraK.2019AmLt}, shows stability similar to the L-scheme while having much faster convergence rates (scaling with $\t$); yet, the convergence is still linear.

In this paper, we investigate a hybrid strategy, dynamically switching between the L-scheme and Newton's method. This utilizes the advantages of both methods: the unconditional stability of the L-scheme, and the quadratic convergence of Newton's method when close to the exact solution. The crucial difference to previous works on hybrid approaches, e.g.~\cite{ListFlorian2016Asoi,bergamaschi1999mixed}, is the adaptive nature of the switch between both linearization methods. A switch from the L-scheme to Newton's method is performed when the iterate is sufficiently close to the solution. This finally allows us to   balance robustness and speed.

The main challenge in implementing this strategy originates from deriving a rigorous switching criteria between the schemes. Since, the {\it a priori} estimates, such as the ones provided in  \cite{radu2006convergence}, involve unknown constants and assume the worst-case scenario, we pursue an {\it a posteriori} estimate-based approach here instead. A rigorous and efficient  {\it a posteriori} estimator for the fully degenerate Richards equation involving linearization errors was derived in \cite{mitra:hal-03328944} in the continuous space-time setting. For the time-discrete problem \eqref{eq: time-discrete Richards equation}, a robust, efficient, and reliable estimator was derived in \cite{Mitra2022lin} using an orthogonal decomposition result dividing the total error into a discretization and a linearization component. Furthermore, its effectiveness was demonstrated numerically. These papers serve as the main inspirations in deriving the {\it a posteriori} based switching criteria  in \Cref{sec: a-posteriori} and an adaptive L-scheme algorithm in Appendix \ref{sec: LtoL}. Nevertheless, since we are only interested in computing the linearization error component, the computation of equilibrated flux will be avoided wherever possible.

The paper is organized as follows. In \Cref{sec:MathForm}, we introduce the mathematical notation, state the assumptions, define the fully-discrete solution, and elaborate on different linearization methods. In \Cref{sec: a-posteriori}, the adaptive switching algorithm is developed. Firstly, a concept of linearization error is introduced along with the derivation of a predictive indicator for linearization error of the next iteration. The adaptive algorithm compares the linearization error with the estimator to   determine the exact switching points. In \Cref{sec: numerics}, three numerical test cases (partially saturated, degenerate, and realistic benchmarks) are presented which illustrate the robustness and computational efficiency of the adaptive scheme compared to the standard Newton's method or the L-scheme.  \Cref{sec: conclusions} contains the conclusions of this work. The paper ends with two appendices, one concerning an adaptive L-scheme and the other on the details of the computation of the equilibrated flux.

\section{Mathematical and numerical formulation}\label{sec:MathForm}

We consider Richards' equation in the space-time domain $\mathcal{G}=\Omega\times[0,T]$, where  $\Omega$ is a bounded domain in $\mathbb{R}^d$ with a Lipschitz continuous boundary $\partial\Omega$, and $T>0$. Let $(\cdot,\cdot)$ and $\|\cdot\|$ be the inner product and norm of the square-integrable functions in $\Omega$, i.e. $L^2(\Omega)$, respectively. Moreover, using common notation from functional analysis, $H^1(\Omega)$ represents the Sobolev space of functions with first-order weak derivatives in $L^2(\Omega)$, and $H_0^1(\Omega)$ its subspace containing functions with vanishing trace at the boundary.

\begin{assumption}\label{as:auxiliary_functions}
For the material properties $\theta$ and $K$, and source term $f$ in \eqref{eq: Richards equation}, the following assumptions are made:
\begin{itemize}
    \item[(a)] The saturation function $\theta(\cdot)$ is Lipschitz continuous and monotonically increasing with $L_\theta$ and $\theta_m\geq 0$ being the Lipschitz constant and the lower bound for the derivative, respectively.  
    
    \item[(b)] The permeability tensor $K:[0,1]\to \R^{d\times d}$  satisfies the uniform (pseudo) ellipticity condition, i.e., for constants $\kappa_M>\kappa_m\geq0$,
    \begin{align*}
       \kappa_m  |\bm{z}|^2\leq   \bm{z}^{\rm T}\,K\,  \bm{z}\leq  \kappa_M  |\bm{z}|^2,\quad  \forall\, \bm{z}\in  \R^d.\; 
    \end{align*}
    Moreover, $(K\circ\theta)$ is Lipschitz continuous, with Lipschitz constant $L_\kappa$.
    \item[(c)] The source function satisfies $f\in C(0,T;L^2(\Om)).$
\end{itemize}
 
\end{assumption}

Note that these assumptions are consistent with the commonly used Brooks-Corey \cite{brooks1966properties} and van Genuchten \cite{vanGenuchtenMTh1980ACEf} parametrizations of the functions $\theta$ and $K$.

\subsection{Time-discretization: Backward Euler}\label{sec:euler}

To discretize the Richards equation in time we consider the backward-Euler time discretization of \eqref{eq: Richards equation}. For this implicit scheme, no CFL conditions need to be satisfied for stability (thus avoiding restrictions on the time step size). Moreover, it does not require higher-order time regularity (unlike the Crank-Nicholson scheme) to converge to the time-continuous solutions. We subdivide the time-interval $[0,T]$ uniformly $N$ times with time step size $\tau=T/N$ and discrete time steps $t_n=\tau n$, where $n\in\left\{1,...,N\right\}$. Then, we look for a sequence $\{\psi^n\}_{n=1}^N$ of functions in $\Om$, satisfying the time-discrete system
\begin{align}\label{eq: time-discrete Richards equation}
    	\frac{\theta(\psi^n)-\theta(\psi^{n-1})}{\t}-\nabla\cdot\left[K(\theta(\psi^n))\nabla(\psi^n+z))\right]=f(t_n).
\end{align}

Denoting $f(t_n)$ by $f^n$ subsequently, a more precise and general definition of the weak solutions of \eqref{eq: time-discrete Richards equation}  is given below. For simplicity, we assume homogeneous Dirichlet boundary condition although our results are valid for Dirichlet and Neumann boundary conditions in general.


\begin{definition}[Backward Euler time-discretization   of \eqref{eq: Richards equation}]
Let $\psi^0\in L^2(\Om)$ be given. Then the sequence $\{\psi^n\}_{n=1}^N\subset H^1_0(\Om)$ is the backward Euler solution of \eqref{eq: Richards equation} if for all $n\in\left\{1,...,N\right\}$, and $v\in H^1_0(\Om)$,
\begin{equation}\label{eq: Galerkin formulation Richards}
	\frac{1}{\t}( \theta(\psi^n)-\theta(\psi^{n-1}),v )+( K(\theta(\psi^n))\nabla(\psi^n+z),\nabla v)=( f^n,\,v).
\end{equation}
    
\end{definition}



\subsection{Space-discretization: Continuous Galerkin finite elements}\label{sec:fem}
We consider the finite element method to discretize \eqref{eq: Galerkin formulation Richards} further in space.
Let $\mathcal{T}_h$ be a triangulation of $\Omega$ into closed $d$-simplices,  where $h:=\max_{E\in\mathcal{T}_h}\left(\mathrm{diam}(E)\right)$ denotes the mesh size. Assuming $\Omega$ is a polygon, the Galerkin finite element space is
\begin{equation}
	V_h=\left\{v_{h}\in H_{0}^{1}(\Omega)|\;v_{h|E} \in\mathcal{P}_p(E), \;T\in\mathcal{T}_h\right\},
\end{equation} 
where $\mathcal{P}_p(E)$ denotes the space of $p$-order polynomials on $E$, $p\in \N$. Then, the fully discrete Galerkin formulation of Richards' equation reads

\begin{definition}[Fully discrete solution of \eqref{eq: Richards equation}]
Let $\psi^0_h :=\psi^0\in L^2(\Om)$. Then the sequence $\{\psi^n_h\}_{n=1}^N\subset V_h$ is the fully discrete solution of \eqref{eq: Richards equation} if for all $n\in\left\{1,...,N\right\}$, and $v_h\in V_h$,
    \begin{equation}
	( \theta(\psi_h^{n})-\theta(\psi_h^{n-1}),v_h)+\tau( K(\theta(\psi_h^n))\nabla(\psi_h^n+z),\nabla v_h)=\tau( f^n,v_h).
	\label{eq: Fully discrete nonlinear Richards}
\end{equation}
\end{definition}



\subsection{Iterative linearization schemes}
To obtain the solution of the nonlinear problem \eqref{eq: Fully discrete nonlinear Richards} an iterative linearization scheme is generally employed. To investigate the trade-off between the stability and speed of such schemes, we focus on two linearization strategies that will be representatives of linearly and quadratically convergent methods with convergence meant in the L$^2$ sense.

\subsubsection{Linearly convergent schemes: The L-scheme}\label{sec: Lin Scheme}


Where the quadratically convergent Newton method utilizes a proper first-order Taylor expansion of the nonlinear terms in \eqref{eq: Fully discrete nonlinear Richards}, the linearly convergent methods that we consider here, only exploit an expansion of the monotone components, i.e. the nonlinear saturation function. Moreover, the expansion does not need to be exact. Consider the following scheme: Given $\psi_h^{n-1},\psi_h^{n,j-1}\in V_h$, find $\psi_h^{n,j}\in V_h$ such that
\begin{align} 
&(\L(\psi_h^{n,j-1})(\psi_h^{n,j}-\psi_h^{n,j-1}),v_h) +\tau( K(\theta(\psi_h^{n,j-1}))\nabla(\psi_h^{n,j}+z),\nabla v_h) \nonumber
\\
&\qquad  =\tau( f^n,v_h) -( \theta(\psi_h^{n,j-1})-\theta(\psi_h^{n-1}),v_h), \label{eq: linear schemes}
\end{align}
for all $v_h\in V_h$, where $\L:\R\to [0,\infty)$ is a predetermined positive weight function, and $j\in \mathbb{N}$ is the iteration index. Observe that, provided $\k_m>0$ in Assumption~\ref{as:auxiliary_functions}, the problem above is linear, monotone, and Lipschitz  with respect to $\psi^{n,j}_h$, and hence a unique weak solution of \eqref{eq: linear schemes} exists. Moreover, if the iteration converges, i.e. if $\psi^{n,j}_h\to \psi^n_h$ strongly in $H^1_0(\Om)$, then $\psi^n_h$ indeed solves \eqref{eq: Fully discrete nonlinear Richards}. There can be many different choices of the function $\L$ which leads to different linearization schemes, see \Cref{tab:linear schemes}. For the rest of this paper, we mainly focus on the case when $\L$ is constant which leads to the widely studied L-scheme. 


\begin{definition}[L-scheme]
   Let $\psi_h^{n-1},\psi_h^{n,0}\in L^2(\Om)$ and $L>0$ be given. Then the L-scheme solves for the sequence $\{\psi_h^{n,j}\}_{j\in \N}\subset V_h$ which satisfies for all iteration indices $j\in \N$, and $v_h\in V_h$
   \begin{equation}\label{eq: L-scheme}
	\begin{aligned}
		&L((\psi_h^{n,j}-\psi_h^{n,j-1}),v_h) +\tau( K(\theta(\psi_h^{n,j-1}))\nabla(\psi_h^{n,j}+z),\nabla v_h)\\
		&=\tau( f^n,v_h) -( \theta(\psi_h^{n,j-1})-\theta(\psi_h^{n-1}),v_h).
	\end{aligned}
\end{equation}\label{def: L-scheme}
\end{definition}

Different choices of $\L$ and the resulting schemes are listed below
\begin{table}[H]
    \centering
    \begin{tabular}{|l|c|}\hline 
         Scheme & $\L(\psi)$  \\[.5em]
         \hline
                  Picard & $0$ \\[.5em]
         Modified Picard \cite{celia_Picard} & $\theta'(\psi)$ \\[.5em]
         J\"ager-Kac\v{u}r \cite{jager1995solution} & $\sup_{\xi\in \R}\frac{\theta(\xi)-\theta(\psi)}{\xi-\psi}$\\[.5em]
         L-scheme \cite{PopI.S.2004Mfef,SlodickaM2002Arae,ListFlorian2016Asoi} & $L>0$ constant\\[.5em]
         Modified L-scheme \cite{MitraK.2019AmLt} & $\theta'(\psi) + M\t$,  $M>0$ constant\\\hline
    \end{tabular}
    \caption{Different linearly convergent schemes \eqref{eq: linear schemes} defined along with their linearization weight function $\L$.}
    \label{tab:linear schemes}
\end{table}

\begin{remark}[Non-constant $L$ for heterogeneous media]
    For the L-scheme, $L$ might not necessarily be a constant, but can be a function of the spatial variable $\bm{x}$. This would be typically the case for heterogeneous media. All the proofs can be adapted to include a spatially dependent $L$, see \cite{both2017etal} where this was done for a splitting scheme for Biot equations.
\end{remark}

It has been shown in \cite[Theorem 1]{ListFlorian2016Asoi} that if $L\geq \frac{1}{2}\sup_{\xi\in \R}\theta'(\xi)$, then the L-scheme iterations converge irrespective of the initial guess
under minor restrictions on the time step size $\t$  and independent of the mesh size. 
However, numerical results in \cite{ListFlorian2016Asoi,MitraK.2019AmLt} reveal that the  convergence of the L-scheme can be relatively slow, depending on the choice of the stabilization parameter $L$, see please the Appendix A for an adaptive L-scheme. One can enhance the convergence speed by computing $L$ using the previous iterates and derivatives. In general, taking $L$ as the Jacobian matrix, would lead to Newton method, this is the reason one can interpret the L-scheme also as a modified Newton method. This is exploited in the modified Picard scheme, first proposed in \cite{celia_Picard}, uses $\L(\psi^{n,j-1})=\theta'(\psi^{n,j-1})$, complying with the first-order Taylor series expansion $\theta(\psi^{n,j})\approx \theta(\psi^{n,j-1}) + \theta'(\psi^{n,j-1})(\psi^{n,j}-\psi^{n,j-1})$. As a result, if converging it requires fewer iterations compared to the L-scheme although the convergence is still linear. Nevertheless, this choice of the $\L$ function may lead to divergence of the scheme for larger time step sizes, as predicted in \cite{radu2006convergence} and observed numerically in \cite{ListFlorian2016Asoi,MitraK.2019AmLt}. In an attempt to resolve this issue, a modified L-scheme was proposed in \cite{MitraK.2019AmLt} that inherits the characteristics of both the L-scheme (except that it is using derivatives and the linear systems are not necessarily well conditioned) and the Picard scheme. The modified L-scheme exhibits increased stability compared to the Picard scheme while retaining its speed. However, the modified L-scheme converges unconditionally under the additional restriction that $\psi^{n,0}_h=\psi^{n-1}_h$ and the discrete time-derivative $(\psi^{n}_h-\psi^{n-1}_h)/\t$ is in $L^\infty(\Om)$. Since the objective of this paper is to start the linearization iterations with a stable scheme, and then switch to a quadratically converging scheme when its convergence can be guaranteed, the rest of the study will be with respect to the L-scheme which is arguably the most stable among the schemes presented in \Cref{tab:linear schemes} and the cheapest in terms of computing time per iteration (due to well-conditioned linear systems and not involving derivatives). Nonetheless, we remark that our methodology generalizes to all other linearly converging iterative methods.

\begin{remark}[Generality of the results]
    Although the analysis of \Cref{sec: a-posteriori} primarily focuses on the switching between L-scheme and the Newton method, the same techniques can be directly extended to cover switching between the schemes in \Cref{tab:linear schemes} and Newton. Moreover, the $L$-adaptive strategy in Appendix \ref{sec: LtoL} can be extended to the modified L-scheme (see \Cref{tab:linear schemes}) to select the parameter $M>0$ adaptively. 
\end{remark}



\subsubsection{Quadratically convergent scheme: The Newton method}
The Newton method uses the first order Taylor series expansions of all the nonlinear functions in \eqref{eq: Richards equation} to ensure quadratic rates of convergence.

\begin{definition}[The Newton method]
   Let $\psi_h^{n-1},\psi_h^{n,0}\in L^2(\Om)$ be given. Then the Newton method solves for the sequence $\{\psi_h^{n,j}\}_{j\in \N}\subset V_h$ which satisfies for all iteration indices $j\in \N$, and $v_h\in V_h$
    \begin{equation}\label{eq: Newtons' method}
	\begin{aligned}
		&(\theta'(\psi_h^{n,j-1})(\psi_h^{n,j}-\psi_h^{n,j-1}),v_h)+\tau( K(\theta(\psi_h^{n,j-1}))\nabla(\psi_h^{n,j-1}+z),\nabla v_h) 
		\\& \quad +\tau\left( (K\circ \theta)'(\psi_h^{n,j-1})\nabla(\psi_h^{n,j-1}+z)(\psi_h^{n,j}-\psi_h^{n,j-1}),\nabla v_h\right)\\
  &\quad =\tau( f^n,v_h)-( \theta(\psi_h^{n,j-1})-\theta(\psi_h^{n-1}),v_h). 
	\end{aligned}
\end{equation}\label{def: Newton}
\end{definition}
However, this comes at the cost of decreased numerical stability as discussed in \Cref{sec: intro}. In the next section we combine the L-scheme and the Newton method in a consistent manner in order to obtain a linerization strategy that is both stable and fast.

\section{A posteriori estimate based adaptive switching between L-scheme and Newton}\label{sec: a-posteriori}

In this section, we develop the switching algorithm between L-scheme and the Newton method using {\it a posteriori} error analysis. For comparing the  errors between different linearization schemes we introduce a uniform notion of linearization errors $\eta_{\rm lin}$ in \Cref{sec:LinErr} based on arguments in \cite{Mitra2022lin}. The idea behind the adaptive algorithm is to start with the L-scheme and derive an estimator $\eta_{\LN}$ in \Cref{sec: LtoN} that predicts from the $j^{\rm th}$ and $(j-1)^{\rm th}$ iterate the linearization error for the next iteration if done using the Newton scheme. If the error is predicted to  decrease, then the iteration switches to Newton. Then another estimator $\eta_{\NL}$ is derived in \Cref{sec: NtoN} which predicts the  linearization error of the next step of the Newton iteration. The algorithm switches back to the L-scheme in case the error is predicted to increase. In fact, we go one step further in Appendix \ref{sec: LtoL} and derive an estimator $\eta_{\LL}$ to predict if the L-scheme itself will converge and to tune the value of $L$ accordingly. Finally, the full algorithm is laid out in \Cref{sec: algo} based on these estimators.

\tikzstyle{terminator} = [rectangle, draw, text centered, rounded corners, minimum height=2em]
\tikzstyle{connector} = [draw, -latex']
\tikzstyle{decision} = [ diamond, draw, text width=3.5em, rounded corners, text centered, node distance=3cm]
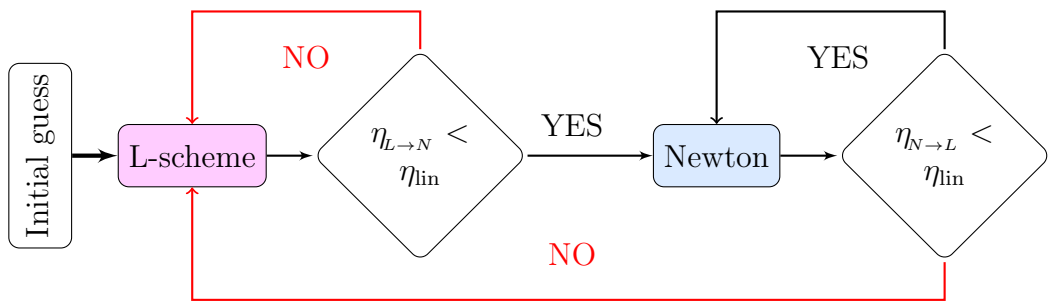
\begin{figure}[H]
\centering
\begin{tikzpicture}[node distance = 3.9cm]

	\node [terminator, fill=colorL1!20] at (2,0) (Lscheme) {L-scheme};
         \node [terminator, rotate=90] (pre) {Initial guess};
	\node [decision, right of=Lscheme, scale=1] at (2,0) (LtoN) {$\eta_{\LN} <\eta_{\rm lin}$};
	\node [terminator, fill=colorNewton!20, right of=LtoN] (Newton) {Newton};
	\node [decision, right of=Newton, scale=1] (NtoL) {$\eta_{\NL}<\eta_{\rm lin}$};
	\node[draw=none,red] at (3.5, 1.3) (no) {NO};
	\node[draw=none] at (7, 0.4) (yes) {YES};
	\node[draw=none] at (10.5, 1.3) (yes2) {YES};
	\node[draw=none,red] at (7, -1.3) (no2) {NO};
	\path [thick,connector] (Lscheme) -- (LtoN);
	\path [thick,connector] (LtoN) -- (Newton);
	\path [thick,connector] (Newton) -- (NtoL);
        \path [ultra thick,connector] (pre)--(Lscheme);
	\draw[thick, ->] (NtoL.north) -- ++(0,0.5)  -| (Newton.north);
	\draw[thick, ->,red] (LtoN.north) -- ++(0,0.5)  -| (Lscheme.north);
	\draw[thick, ->,red] (NtoL.south) -- ++(0,-0.5)  -| (Lscheme.south);
\end{tikzpicture}
\caption{Flowchart of Adaptive switching algortihm between L-scheme and Newton's method.}
\end{figure}

\subsection{Linearization errors and iteration-dependent energy norms}\label{sec:LinErr}
In \cite{Mitra2022lin} it is shown that the total numerical error corresponding to a finite element-based linearization scheme can be orthogonally decomposed into a discretization component and a linearization component if the errors are computed using an iteration-dependent energy norm (for linearly convergent schemes in \Cref{tab:linear schemes} this is just the energy norm invoked by the symmetric bilinear form associated with the unknown $\psi^{n,j}_h$ in \eqref{eq: linear schemes}).  Here, we are only interested in the linearization component which is defined as the difference between successive iterates in the aforementioned energy norm, i.e.,
\begin{align}\label{eq: linearization error}
    \eta_{\rm lin}^j:= \norm{\psi^{n,j}_h-\psi^{n,j-1}_h}_{\L,\psi^{n,j-1}_h},
\end{align}
where $\norm{\cdot}_{\L,\psi^{n,j-1}_h}$ represents the particular  $H^1$ equivalent-norm defined using the iterate $\psi^{n,j-1}_h$ and associated with the linearization scheme denoted by $\L$. The fully computable estimator $\eta_{\rm lin}^j$ encapsulates the entirety of the linearization error, as shown in Section 5 of \cite{Mitra2022lin}, and hence, will be used as its sole measure in the subsequent sections. We mention explicitly the energy norms of the two schemes that are discussed: With reference to \Cref{def: L-scheme}, the energy norm for L-scheme is defined as 
\begin{equation}\label{eq:energy norm L}
		\norm{\xi}_{L,\psi_h^{n,j-1}}:=\left(\int_\Omega L\xi^2+\t \left|K(\theta(\psi_h^{n,j-1}))^{\frac{1}{2}}\nabla\xi\right|^2\right)^\frac12
\end{equation}
 for all $\xi\in H_0^1(\Omega)$,  and with reference to \Cref{def: Newton} the norm for the Newton method is
	\begin{equation}\label{eq:energy norm Newton}
		\norm{\xi}_{N,\psi_h^{n,j-1}}:=\left(\int_\Omega \theta'(\psi_h^{n,j-1})\, \xi^2+\tau |K(\theta(\psi_h^{n,j-1}))^{\frac12}\nabla\xi|^2\right)^\frac12.
	\end{equation}

\subsection{L-scheme to Newton switching estimate}\label{sec: LtoN}
For some $i\in \N$, let the sequence $\{\psi^{n,j}_h\}_{j=1}^i\subset V_h$ be obtained using the L-scheme \eqref{eq: L-scheme}, and in the $(i+1)^{\rm th}$-iteration we want to test for switching to the Newton scheme. Let $\tilde{\psi}^{n,i+1}_h\in V_h$ be the solution of the Newton scheme \eqref{eq: Newtons' method} having $\psi^{n,i}_h$ as the previous iterate. 
In this section, we will assume the following:
\begin{assumption}[Convection term is not dominant]\label{as:convection}
For a given $i\in \N$, there exists a constant $C^i_N\in [0,2)$ such that
\begin{align}\label{eq:convection}
    \t |K(\theta(\psi_h^{n,i}))^{-\frac{1}{2}}(K\circ\theta)'(\psi_h^{n,i})\nabla(\psi_h^{n,i}+z)|^2 \leq (C_N^i)^2 \theta'(\psi^{n,i}_h),
\end{align}
a.e. in $\Om$.
\end{assumption}

The assumption above is also required to show the coercivity of the linear problem  \eqref{eq: Newtons' method} for $j=i+1$, and hence, to show the existence of solution $\Tilde{\psi}^{n,i+1}_h$. Observe that, since $\psi^{n,i}_h$ is known, the constant $C_N^i$ is fully computable. Additionally, it is smaller than 2  if the numerical flux is bounded, and $\t$ is small. Notably, the estimate holds even in the degenerate case when $\theta'(\psi^{n,i}_h)=0$, since the left-hand side has $(\theta'(\psi^{n,i}_h))^2$. 
To cover the degenerate case, we also introduce the concept of an equilibrated flux.

\begin{definition}[Equilibrated flux $\efl$ for degenerate regions]
For a pre-determined $\eps>0$, let  $\calT_{\rm deg}^{i,\eps}:= \{K\in \calT_h: \inf \theta'(\psi^{n,i}_h)<\eps \text{ in } K \}$. Let $\Pi_{h}: L^2(\Om)\to {\cal P}_p(\calT_h) $ be the $\mathcal{P}_p$ projection operator, i.e. $(\Pi_h u,v_h)=(u,v_h)$ 
 for all $u\in L^2(\Om)$ and $v_h\in {\cal P}_p(\calT_h)$. Moreover, let $\bm{{\rm RT}}_p(\calT_h)$ be the $p^{\rm th}$-order Raviart-Thomas space on $\calT_h$, i.e., $\bm{\sigma}\in \bm{{\rm RT}}_p(\calT_h)$ implies $\bm{\sigma}|_K\in  ({\cal P}_p(K))^d + \bm{x} {\cal P}_p(K)$ for all $K\in \calT_h$.
    Then, we define $\efl \in \bm{{\rm RT}}_p(\calT_h)\,\cap\, \bm{H}({\rm div},\Om)$ as
    \begin{align}\label{eq:equilibrated flux L}
        \del \cdot \efl=       \begin{cases}
            \frac{1}{\t}\Pi_{h} (L(\psi^{n,i}_h-\psi^{n,i-1}_h)-(\theta(\psi^{n,i}_h)-\theta(\psi^{n,i-1}_h))) &\text{ in } \calT_{\rm deg}^{i,\eps},\\
            0 &\text{ otherwise }.
        \end{cases}     
    \end{align}\label{def:equilibrated flux} 
\end{definition}

We defer to Appendix \ref{sec: eq flux} for discussions on how to compute $\efl$ in practice. Then, we have the following result.

\begin{proposition}[Error control of L-scheme to Newton switching step]\label{prop:LtoN}
For a given $\psi^{n,0}_h,\, \psi^{n-1}_h\in V_h$, let $\{\psi^{n,j}_h\}^{i}_{j=1}\subset V_h$ solve \eqref{eq: L-scheme} for some $i\in \mathbb{N}$. Let $\Tilde{\psi}^{n,i+1}_h\in V_h$ be the solution of \eqref{eq: Newtons' method} with the previous iterate $\psi^{n,i}_h$. Recall \Cref{def:equilibrated flux}.  Then, under the Assumptions \ref{as:auxiliary_functions}--\ref{as:convection}, one has
$$
\norm{\Tilde{\psi}^{n,i+1}_h-\psi^{n,i}_h}_{N,\psi^{n,i}_h}\leq \eta_{\LN}^i,
$$
where,
$$
\eta_{\LN}^i := \tfrac{2}{2-C_N^i}\left(\left[\eta_{\LN}^{i, \rm poten}\right]^2 + \t \left[\eta_{\LN,2}^{i, \rm flux}\right]^2\right)^{\frac{1}{2}}
$$
with
\begin{align*}
\eta_{\LN}^{i, \rm poten}&:= \left\|\theta'(\psi^{n,i}_h)^{-\frac{1}{2}}\left(L\left(\psi^{n,i}_h-\psi^{n,i-1}_h\right)-\left(\theta(\psi^{n,i}_h)-\theta(\psi^{n,i-1}_h)\right)\right)\right\|_{\calT_h\setminus \calT_{\rm deg}^{i,\eps}},\\
\eta_{\LN}^{i, \rm flux}&:= \left\|K(\theta(\psi_h^{n,i}))^{-\frac{1}{2}} \left[\left(K(\theta(\psi_h^{n,i}))-K(\theta(\psi_h^{n,i-1}))\right)\nabla\left( \psi^{n,i}_h+z\right) + \bm{\sigma}^i_L\right]\right\|.
\end{align*}
\end{proposition}


\begin{proof}
Observe from \eqref{eq: Newtons' method} that  $\d \psi^{i+1}_h:= \Tilde{\psi}^{n,i+1}_h-\psi^{n,i}_h\in V_h$ satisfies
\begin{align}\label{eq:NewtonIth}
		&(\theta'(\psi_h^{n,i})\d \psi^{i+1}_h,v_h)+\tau( K(\theta(\psi_h^{n,i}))\nabla \d \psi^{i+1}_h,\nabla v_h) \nonumber
		\\& \quad +\tau\left( (K\circ\theta)'(\psi_h^{n,i})\nabla(\psi_h^{n,i}+z)\,\d\psi^{i+1}_h,\nabla v_h\right)\nonumber\\
  &\quad =\tau( f^n,v_h)-( \theta(\psi_h^{n,i})-\theta(\psi_h^{n-1}),v_h) -\tau( K(\theta(\psi_h^{n,i}))\nabla \psi^{i}_h,\nabla v_h), 
\end{align}
for all $v_h\in V_h$.
    Inserting the test function $v_h= \d \psi^{i+1}_h$ in \eqref{eq:NewtonIth}, one has
    \begin{subequations}
\begin{align}
    &\norm{\d \psi^{i+1}_h}_{N,\psi_h^{n,i}}^2 \overset{ \eqref{eq:energy norm Newton}}= \int_\Omega \left (\theta'(\psi_h^{n,i})|\d \psi^{i+1}_h|^2+\tau | K(\theta(\psi_h^{n,i}))^{\frac{1}{2}}\nabla \d \psi^{i+1}_h |^2 \right)\nonumber\\
    &\;\;\overset{\eqref{eq:NewtonIth}}=\underbrace{-\tau\left( (K\circ\theta)'(\psi_h^{n,i})\nabla(\psi_h^{n,i}+z)\,\d\psi^{i+1}_h,\nabla \d \psi^{i+1}_h\right)}_{=: T_1}\nonumber \\
    &\quad+ \underbrace{\tau( f^n,\d \psi^{i+1}_h)-( \theta(\psi_h^{n,i})-\theta(\psi_h^{n-1}),\d \psi^{i+1}_h) -\tau( K(\theta(\psi_h^{n,i}))\nabla( \psi^{n,i}_h+z),\nabla \d \psi^{i+1}_h)}_{=:T_2}.
\end{align}
Calling $\bm{\sigma}^i= (K\circ\theta)'(\psi_h^{n,i})\nabla(\psi_h^{n,i}+z)$ for brevity, we estimate that
\begin{align}\label{eq:estimate T1}
    T_1&:= -\t(\bm{\sigma}^i\d\psi^{i+1}_h,\nabla \d \psi^{i+1}_h)\nonumber\\
    &\leq  \left(\t\int_\Om |K(\theta(\psi_h^{n,i}))^{-\frac{1}{2}} \bm{\sigma}^i|^2 (\d\psi^{i+1}_h)^2 \right)^{\frac{1}{2}} \left( \t\int_\Om|K(\theta(\psi_h^{n,i}))^{\frac{1}{2}} \nabla \d \psi^{i+1}_h|^2\right)^{\frac{1}{2}} \nonumber\\
    &\overset{\eqref{eq:convection}}\leq C_N^i \left(\int_\Om \theta'(\psi^{n,i}_h) (\d\psi^{i+1}_h)^2 \right)^{\frac{1}{2}} \left( \t\int_\Om|K(\theta(\psi_h^{n,i}))^{\frac{1}{2}} \nabla \d \psi^{i+1}_h|^2\right)^{\frac{1}{2}}\nonumber\\
    &\leq \frac{C_N^i}{2} \int_\Om \left( \theta'(\psi^{n,i}_h) |\d\psi^{i+1}_h|^2 + \t |K(\theta(\psi_h^{n,i}))^{\frac{1}{2}} \nabla \d \psi^{i+1}_h|^2 \right) \nonumber\\
    &= \frac{C_N^i}{2} \norm{\d\psi^{i+1}_h}^2_{N,\psi^{n,i}_h}.
\end{align}
For estimating the last term, we observe from the divergence theorem that
\begin{align*}
    &-(\efl,\del \d\psi^{i+1}_h)=(\del\cdot\efl,\d\psi^{i+1}_h) \\
    &\quad \overset{\eqref{eq:equilibrated flux}}=\tfrac{1}{\t}(\Pi_{h} (L(\psi^{n,i}_h-\psi^{n,i-1}_h)-(\theta(\psi^{n,i}_h)-\theta(\psi^{n,i-1}_h))),\d\psi^{i+1}_h)_{\calT_{\rm deg}^{i,\eps}}\\
    &\quad = \tfrac{1}{\t}(L(\psi^{n,i}_h-\psi^{n,i-1}_h)-(\theta(\psi^{n,i}_h)-\theta(\psi^{n,i-1}_h)),\d\psi^{i+1}_h)_{\calT_{\rm deg}^{i,\eps}}
\end{align*}
The last equality follows from the definition of the projection operator $\Pi_h$ and $\d \psi^{i+1}_h\in V_h\subset \calP_p(\calT_h)$.
Using this result, along with \eqref{eq: L-scheme} and $\d \psi^{i+1}_h\in V_h$, one has
\begin{align}
    T_2&:= \tau( f^n,\d \psi^{i+1}_h)-( \theta(\psi_h^{n,i})-\theta(\psi_h^{n-1}),\d \psi^{i+1}_h) -\tau( K(\theta(\psi_h^{n,i}))\nabla \psi^{i}_h,\nabla \d \psi^{i+1}_h)\nonumber\\
    &\overset{\eqref{eq: L-scheme} }= (L(\psi_h^{n,i}-\psi_h^{n,i-1})- (\theta(\psi_h^{n,i})-\theta(\psi_h^{n,i-1})),\d \psi^{i+1}_h)\nonumber\\
    &\qquad -\tau(( K(\theta(\psi_h^{n,i}))-K(\theta(\psi_h^{n,i-1})))\nabla( \psi^{n,i}_h+z),\nabla \d \psi^{i+1}_h)\nonumber\\
   &= (L(\psi_h^{n,i}-\psi_h^{n,i-1})- (\theta(\psi_h^{n,i})-\theta(\psi_h^{n,i-1})),\d \psi^{i+1}_h) +\t (\efl,\del \d\psi^{i+1}_h)\nonumber\\
    &\qquad -\tau(( K(\theta(\psi_h^{n,i}))-K(\theta(\psi_h^{n,i-1})))\nabla( \psi^{n,i}_h+z) + \efl,\nabla \d \psi^{i+1}_h)\nonumber\\
    &= (L(\psi_h^{n,i}-\psi_h^{n,i-1})- (\theta(\psi_h^{n,i})-\theta(\psi_h^{n,i-1})),\d \psi^{i+1}_h)_{\calT_h\setminus \calT_{\rm deg}^{i,\eps}}\nonumber\\
    &\qquad -\tau(( K(\theta(\psi_h^{n,i}))-K(\theta(\psi_h^{n,i-1})))\nabla( \psi^{n,i}_h+z) + \efl,\nabla \d \psi^{i+1}_h)\nonumber\\
   &\overset{\eqref{eq:equilibrated flux L}}\leq  (\theta'(\psi^{n,i}_h)^{-\frac{1}{2}}(L(\psi_h^{n,i}-\psi_h^{n,i-1})- (\theta(\psi_h^{n,i})-\theta(\psi_h^{n,i-1}))),\theta'(\psi^{n,i}_h)^{\frac{1}{2}}\d \psi^{i+1}_h)_{\calT_h\setminus \calT_{\rm deg}^{i,\eps}}\nonumber\\
   &\qquad + \t [\eta_{\LN}^{i, \rm flux}]\, \|K(\psi^{n,i}_h)^{\frac{1}{2}}\del \d \psi^{i+1}_h\| \nonumber\\
   &\leq [\eta_{\LN}^{i, \rm poten}]\cdot  \|\theta'(\psi^{n,i}_h)^{\frac{1}{2}}\d \psi^{i+1}_h\| + \sqrt{\t}\,[\eta_{\LN}^{i, \rm flux}]\cdot \sqrt{\t} \|K(\psi^{n,i}_h)^{\frac{1}{2}}\del \d \psi^{i+1}_h\|.
\end{align}\label{eq:LNswicth}
    \end{subequations}
    Combining \eqref{eq:LNswicth}, using the Cauchy-Schwarz inequality along with the definition of $\eta_{\LN}^i$, one has the estimate.
\end{proof}

\subsection{Newton to L-scheme switching estimate}\label{sec: NtoN}
Assuming that the L-scheme converges unconditionally, after switching to Newton we would want to switch back to the L-scheme only if linearization error of the Newton scheme increases with iterations. Similar to before, we can estimate if this is going to happen in the $(i+1)^{\rm th}$-step, purely from the iterates up to the $i^{\rm th}$-step. For this purpose, we introduce another equilibrated flux.

\begin{definition}[Equilibrated flux $\efl$ for degenerate regions (Newton scheme)]
Recalling \Cref{def:equilibrated flux}, we define $\efn \in \bm{{\rm RT}}_p(\calT_h)\,\cap\, \bm{H}({\rm div},\Om)$ as
    \begin{align}\label{eq:equilibrated flux}
        \del \cdot \efn=       \begin{cases}
            \frac{1}{\t}\Pi_{h} (\theta'(\psi^{n,i}_h)(\psi^{n,i}_h-\psi^{n,i-1}_h)-(\theta(\psi^{n,i}_h)-\theta(\psi^{n,i-1}_h))) &\text{ in } \calT_{\rm deg}^{i,\eps},\\
            0 &\text{ otherwise }.
        \end{cases}     
    \end{align}\label{def:equilibrated flux Newton} 
\end{definition}

The corresponding result mirroring \Cref{prop:LtoN} is

\begin{proposition}[Error control of Newton to Newton step]\label{prop:NtoN}
For a given $\psi^{n,0}_h,\, \psi^{n-1}_h\in V_h$, let $\{\psi^{n,j}_h\}^{i+1}_{j=1}\subset V_h$ solve \eqref{eq: Newtons' method} for some $i\in \mathbb{N}$. Then, under Assumptions \ref{as:auxiliary_functions}--\ref{as:convection}, one has
$$
\norm{\psi^{n,i+1}_h-\psi^{n,i}_h}_{N,\psi^{n,i}_h}\leq \eta_{\NL}^i,
$$
where $$\eta_{\NL}^i :=\tfrac{2}{(2-C_N^i)}\left([\eta_{\NL}^{i,{\rm poten}}]^2 + \t [\eta_{\NL}^{i,{\rm flux}}]^2\right)^{\frac{1}{2}}$$
with \begin{align*}
		&\eta_{\NL}^{i,{\rm poten}}:= \|\theta'(\psi^{n,i}_h)^{-\frac{1}{2}}(\theta'(\psi^{n,i-1}_h)(\psi^{n,i}_h-\psi^{n,i-1}_h)-(\theta(\psi^{n,i}_h)-\theta(\psi^{n,i-1}_h)))\|_{\calT_h\setminus \calT_{\rm deg}^{i,\eps}},\\
		&\eta_{\NL}^{i,{\rm flux}}:= \left\| \begin{matrix}
			\left[(K(\theta(\psi_h^{n,i}))-K(\theta(\psi_h^{n,i-1})))\del (\psi^{n,i}_h+z) 
   \phantom{AAAAAAAAAAA}\right.\\\left.- (K\circ\theta)'(\psi_h^{n,i-1})(\psi_h^{n,i}-\psi_h^{n,i-1})\del (\psi^{n,i-1}_h+z)
			)\right]K(\theta(\psi_h^{n,i}))^{-\frac{1}{2}}\\
		+ K(\theta(\psi_h^{n,i}))^{-\frac{1}{2}}\efn
		\end{matrix}   \right\|.
\end{align*}
\end{proposition}

The proof is identical to the proof of \Cref{prop:LtoN} and hence is left for the avid reader.

\begin{remark}[Effectivity of the estimators $\eta_{\LN}^i$ and $\eta_{\NL}^i$]\label{rem: effectivity}
    The estimators  $\eta_{\LN}^i$ and $\eta_{\NL}^i$ predict the linearization error $\eta_{\rm lin}^{i+1}$ of the $(i+1)^{\rm th}$ iteration if done using the Newton scheme \eqref{eq: Newtons' method}. In the cases where the iteration is done indeed using the Newton scheme, the sharpness of the estimate can be measured using the \textbf{effectivity index}, i.e., if $(i+1)^{\rm th}$ iteration is Newton then
\begin{align}\label{eq: eff ind}
    \text{(Eff. Ind.)}_i:=\begin{cases}
        \eta_{\LN}^i/\eta_{\rm lin}^{i+1} &\text{ if } i^{\rm th} \text{ iteration is L-scheme},\\
        \eta_{\NL}^i/\eta_{\rm lin}^{i+1} &\text{ if } i^{\rm th} \text{ iteration is Newton}.        
    \end{cases}
\end{align}
Observe that it is always greater than 1 due to \Cref{prop:LtoN,prop:NtoN} and an effectivity index close to 1 implies a sharp estimate. The estimators are expected to be quite accurate since mainly the Cauchy-Schwarz inequality is used to derive them, except for estimate \eqref{eq:estimate T1} where the term $T_1$ is bounded above using the global approximation in Assumption \ref{as:convection}. This expected sharpness is shown to be the case through the numerical experiments of \Cref{sec: numerics}, see in particular \Cref{fig:switching indictor ex1,fig:switching indictor ex2}
\end{remark}

\subsection{A-posteriori estimate based adaptive linearization algorithm}\label{sec: algo}
With the above estimates in mind, we propose a switching algorithm between the L-scheme and the Newton method.  The linearization scheme used at iteration $j=i+1$ should be Newton if the  linearization error, predicted by the estimators $\eta^i_{\LN}$ and $\eta^i_{\NL}$, is smaller than the linearization error $\eta_{\rm lin}^i$ of the $i^{\rm th}$ step,  see \eqref{eq: linearization error}. However, to optimize the  algorithm we take a few numerical considerations into account first.

\subsubsection{Computational considerations}

To speed up the computations of this switching criteria, we make a few more reductions
\begin{itemize}
    \item \textbf{[Equilibrated flux]}  If the saturated domain is much smaller than the unsaturated domain, then we take $\efl=\efn=0$.
    \item  \textbf{[Switching condition]}  The condition $\eta_{L\to N}^i \leq \eta_{\rm lin}^i$ might be difficult to satisfy if the estimators are not sharp (see \Cref{rem: effectivity}), and even when it is satisfied it might require large values of $i$. Hence, to expedite the switching between L-scheme and Newton, we will use the criteria $\eta_{L\to N}^i< C_{\rm tol} \,\eta_{\rm lin}^i$ for a constant $C_{\rm tol}>1$. 
\end{itemize}

\subsubsection{Adaptive linearization algorithm}
 Under these considerations we propose the following adaptive algorithm:

\begin{algorithm}[H]
	\caption{L-scheme/Newton a-posteriori switching}\label{alg: a-posteriori}
	\begin{algorithmic}
		\Require $\boldsymbol{\psi}^{n,0}\in L^2(\Om)$ as initial guess. 
        \Ensure Scheme=\fbox{L-scheme}, $C_{\rm tol}=1.5$
	\For{i=1,2,..}
  \If{Scheme=\fbox{L-scheme}} 
  \State
   Compute iterate using  \fbox{L-scheme}, i.e., \eqref{eq: L-scheme}
        \If{$C_N^i\geq 2$}  \textbf{continue}.        
		\ElsIf{$\eta_{\LN}^i\leq C_{\rm tol} \eta^i_{\rm lin}$}
		  \State Set Scheme=\fbox{Newton}
		\EndIf
  \Else 
  		\State \  Compute iterate using \fbox{Newton}, i.e., \eqref{eq: Newtons' method}
        \If{$\eta_{\NL}^i> \eta^i_{\rm lin}$}
				\State Set Scheme=\fbox{L-scheme}
			\EndIf
   \EndIf
		\EndFor

	\end{algorithmic}
\end{algorithm}

\begin{remark}[Combining L-scheme adaptivity]
    In Appendix \ref{sec: LtoL}, we further propose an algorithm to adaptively select $L$ in order to expedite the convergence of the L-scheme. This can directly be implemented in conjunction to \Cref{alg: a-posteriori} to improve the convergence speed of the composite scheme. Nevertheless, we have refrained from combining these schemes for the ease of presentation.
\end{remark}

\begin{remark}[Computational cost of the estimators]
    In the non-degenerate case, the quantities $C_N^i$, $\eta^i_{\LN}$ and $\eta^i_{\NL}$, can be directly computed from the iterates $\psi^{n,i}_h$ and $\psi^{n,i-1}_h$ by inserting $\efl=\efn=0$, see \Cref{prop:LtoN,prop:NtoN}. Hence, the cost of computing the estimators is small in comparison to the cost of the iterations. Since the L-scheme iterations are less expensive than the Newton iterations, the L/N scheme generally performs better or similarly to the Newton scheme time-wise. This is evident from the numerical experiments, e.g. see \Cref{fig: CPU time ex1}. In the degenerate case, global computation are required for computing $\efl$ and $\efn$ if they are used. We discuss the computation of these equilibrated fluxes in Appendix \ref{sec: eq flux} and their computation can be made relatively inexpensive by precomputing the associated stiffness matrices. The computational cost for the estimators can be reduced even further by evaluating them only for selected iterations. Nevertheless, we do not pursue this option for the sake of simplicity.
\end{remark}



\section{Numerical results}\label{sec: numerics}

In this section, we perform several numerical examples that demonstrate the robustness and efficiency of the proposed algorithm for switching between Newton's method and the L-scheme. This is done through careful comparison between the switching algorithm, hereafter called the L/N-scheme, the standard Newton method and the L-scheme. It is important to note that the L-scheme includes a tuning parameter that significantly affects the performance of the method. As a remedy, we choose two different values, $L_1$ and $L_2$ in the performance comparison. Here, $L_1$ is a quasi-optimal choice of tuning parameter and will be defined for each specific subproblem, see Table~\ref{table: example parameters}, and $L_2=\sup\left\{\theta'\left(\psi\right)\right\}$. For the L/N-scheme, $L_1$ is always chosen for the L-scheme iterations.

To measure the performance of each separate method, we examine both the number of iterations and computational time that they require to satisfy the stopping criterion $$\norm{\psi_h^{n,j}-\psi_h^{n,j-1}}_{\L,\psi_h^{n,j-1}}<10^{-7},$$ where $\norm{\cdot}_{\L,\psi_h^{n,j-1}}$ is the iteration and linearization-dependent energy norm for the pressure head, with $\L\in \{L,N\}$. Here, the computational time covers the entire simulations and all experiments were performed on an Acer Swift 3, with an Intel core i7-1165G7-processor.   

In total, three different test cases for the numerical experiments are considered:
\begin{itemize}
    \item Test case 1: The first test case is taken from \cite{IllianoDavide2021Isfs}, although it is modified in the sense that we disregard surfactant transport. Here, the flow is always partially saturated.
    \item Test case 2: The second test case can be found in \cite{ListFlorian2016Asoi}, and it considers extraction/injection above the water table.
    \item Test case 3: The final test case is a known benchmark problem that is studied in \cite{ListFlorian2016Asoi,Knabner1987,Haverkamp1977,schneid2000hybrid}. Here, a time-dependent Dirichlet boundary condition is used to describe the recharge of a groundwater reservoir from a drainage trench.
\end{itemize} 

For all test cases, the van Genuchten-Mualem parametrization \cite{vanGenuchtenMTh1980ACEf} is used to describe the relation between the saturation, the pressure head and the permeability,
\begin{equation}\label{eq: Van Genutchen}
	\begin{aligned}
		\theta(\psi)&=\begin{cases}
			\theta_R+(\theta_S-\theta_R)\left[\frac{1}{1+(-\alpha\psi)^n}\right]^{\frac{n-1}{n}},\quad &\psi\leq 0,\\
			\theta_S,\quad&\psi>0,
		\end{cases}\\
		K(\Theta(\psi))&=\begin{cases}
			K_s\, (\Theta(\psi))^{\frac{1}{2}}\left[1-\left(1-\Theta(\psi)^{\frac{n}{n-1}}\right)^{\frac{n-1}{n}}\right]^2,\quad &\psi\leq 0,\\
			K_s,\quad &\psi>0.\\
		\end{cases}
	\end{aligned}
\end{equation}
Here,
\begin{equation*}
	\Theta(\psi)=\frac{\theta(\psi)-\theta_R}{\theta_S-\theta_R},
\end{equation*}
with $\theta_S$ and $\theta_R$ being the water volume and the residual water content respectively, $K_s$ the hydraulic conductivity of the fully saturated porous medium, and $\alpha$ and $n$ soil related parameters.

In all of the test-cases, triangular linear conforming finite elements with mesh diameter $h$ are applied together with the implicit Euler time-discretization with time step size $\tau$, as described in~\Cref{sec:fem,sec:euler}. The mesh diameter $h$ and time step size $\tau$ vary between the different experiments and will be specified for each individual experiment. We note that the numerical experiments are expected to perform equivalently for other spatial discretization methods such as the Raviart-Thomas mixed finite elements or discontinuous Galerkin finite elements.

The finite element implementation is Python based  and uses the simulation toolbox PorePy \cite{PorePyPaper} for grid management. It is available for download at \url{https://github.com/MrShuffle/RichardsEquation/releases/tag/v1.0.1}. 


\begin{table}[H]
\setlength\tabcolsep{0pt} 
\begin{tabular*}{\textwidth}{@{\extracolsep{\fill}} l *{4}{d{2.8}} }
\toprule
 
 \multicolumn{2}{l}{Parameters}& \multicolumn{1}{c}{Test case 1} & \multicolumn{1}{c}{Test case 2 } & \multicolumn{1}{c}{Test case 3 } \\ 
\midrule
\multicolumn{5}{l}{van Genuchten-Mualem}\\
&\theta_R     & 0.026 &0.026& 0.131 \\

&\theta_S     & 0.42 &0.42& 0.396  \\
&K_S
              & 0.12 &0.12& 4.96\cdot 10^{-2}   \\
&\alpha &0.551&0.95& 0.423\\
&n &2.9&2.9&2.06\\
\midrule
\multicolumn{5}{l}{L-scheme}\\
&L_1 & 0.1 &0.15& 3.501\cdot 10^{-3}\\
&L_2 = L_\theta & 0.136 &0.2341& 4.501\cdot 10^{-3}\\
\bottomrule
\end{tabular*}
\caption{Parameter values for all test cases. The parameters are presented in column format, where each column corresponds to the parameters for the specified test case.}
\label{table: example parameters}
\end{table}

\subsection{Test case 1: Strictly unsaturated medium}

In this test case, we consider a strictly unsaturated porous medium, and use the van Genuchten-Mualem parametrization that is described by parameters from \Cref{table: example parameters}. The test case is heavily inspired by \cite{IllianoDavide2021Isfs}, and the domain is given by $\Omega = \Omega_{1}\cup\Omega_{2}$, where $\Omega_{1}=[0,1]\times[0,1/4]$ and $\Omega_{2}=[0,1]\times(1/4,1]$. We consider the time interval $[0,T]$, where $T = \tau$ varies with choice of time step size $\tau$, as we only take one time step. As initial condition, we choose the pressure head
$$\boldsymbol{\psi}^0(x,z)=\begin{cases}-z-1/4 &(x,z)\in  \Om_1\\ -4 &(x,z)\in  \Om_2, \end{cases}$$ where $x$ represents the positional variable in the horizontal direction and $z$ in the vertical direction. 
A Dirichlet boundary condition is imposed at the top boundary that complies with the initial condition. For the rest of the boundary no-flow boundary conditions are used, and the following source term is applied
$$
f(x,z)=\begin{cases} 0 &(x,z)\in \Om_1\\ 0.06\cos\left(\frac{4}{3}\pi(z)\right)\sin\left(x\right)
&(x,z)\in \Om_2.\end{cases}$$
The solution after one time step with time step size $\tau = 1,$ is given in Figure~\ref{fig:unsaturatedpressure}.


\begin{figure}[H]
    \centering
    \includegraphics[width=8cm]{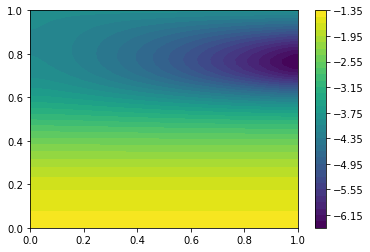}
    \caption{Test case 1: Strictly unsaturated medium. Pressure head $\psi$ at final time $T=1$.}
    \label{fig:unsaturatedpressure}
\end{figure}

\subsubsection{Comparison of convergence properties.}



Here, we discuss the performance and convergence properties of the newly proposed L/N-scheme and compare it to the Newton method and the L-scheme.
In~\Cref{fig: convergence ex1}, the number of iterations for different choices of the mesh size parameters, with time step size $\tau = 0.01$ are presented. As expected the L-scheme is robust and converges in each scenario, for both $L_1$ and $L_2$. Newton’s method, however, only converges for sufficiently coarse meshes. Yet, when converging, it converges in fewer iterations than the L-scheme. Finally, the hybrid L/N method converges in as few if not fewer iterations as the Newton method (when it converges) and converges robustly, and in far fewer iterations than the L-scheme for the other mesh sizes.

Furthermore, a similar experiment is performed for a fixed mesh size $h = \sqrt{2}/40$, and varying time step sizes, see \Cref{fig: number of itr ex1}.
For larger time step sizes the Newton method diverges, while the other methods converge robustly. Again the L/N-scheme converges with the performance expected of Newton's method, in addition to being as robust as the L-scheme. We highlight the enormous difference in the number of iterations for the largest time step size $\tau= 1$ in \Cref{fig: number of itr ex1}. 



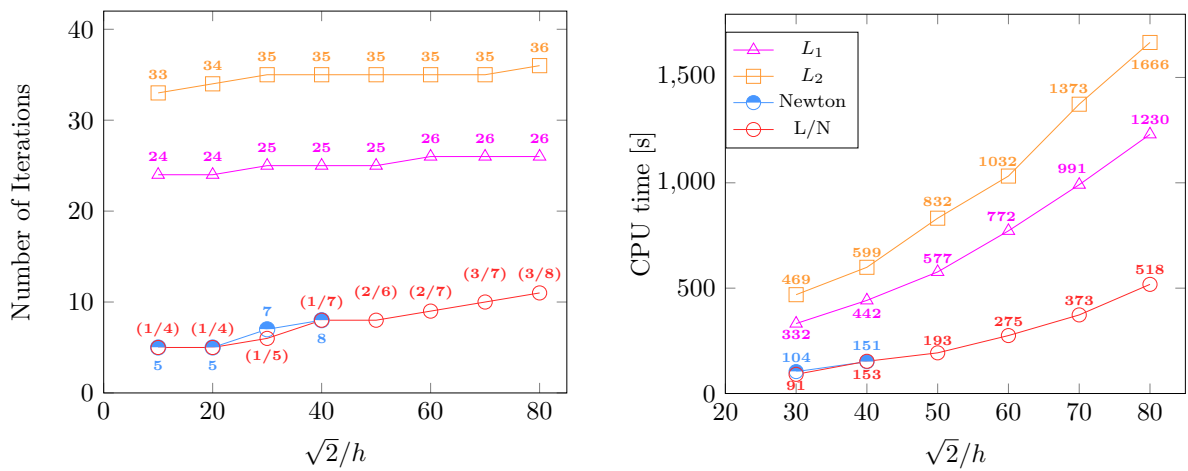
\begin{figure}[H]
\centering
    \subfloat[][Total number of iterations. The numbers in the red parentheses correspond to (number of L-scheme iterations/number of Newton iterations).\label{fig: convergence ex1}]{\resizebox{0.48\textwidth}{!}{
\begin{tikzpicture}[thick, scale=1]
\tikzstyle{every node}=[font=\small]
\begin{axis}
[
    xlabel={\scriptsize Time},
    xmin = 0,
    xmax = 85,
    ymin = 0,
    ymax = 42,
xlabel={$\sqrt{2}/h$}, ylabel=Number of Iterations] 
\addplot[color=colorL1,mark=triangle,mark options={scale=1.5}] 
coordinates { (10,24) (20,24) (30,25) (40,25) (50,25) (60,26) (70,26) (80,26) }; 
\addplot[color=colorL2,mark=square,mark options={scale=1.5}] 
coordinates { (10,33) (20,34) (30,35) (40,35) (50,35) (60,35) (70,35) (80,36) }; 
\addplot[color=colorNewton,mark=halfcircle*,mark options={scale=1.6}] 
coordinates { (10,5) (20,5) (30,7) (40,8) }; 
\addplot[color=colorLN,mark=o,mark options={scale=1.5}] 
coordinates { (10,5) (20,5) (30,6) (40,8) (50,8) (60,9) (70,10) (80,11) };

\node[font=\bf,color=colorL1] at (100,260){\tiny 24};
\node[font=\bf,color=colorL1] at (200,260){\tiny 24};  
\node[font=\bf,color=colorL1] at (300,270){\tiny 25};
\node[font=\bf,color=colorL1] at (400,270){\tiny 25};
\node[font=\bf,color=colorL1] at (500,270){\tiny 25};
\node[font=\bf,color=colorL1] at (600,280){\tiny 26};
\node[font=\bf,color=colorL1] at (700,280){\tiny 26};
\node[font=\bf,color=colorL1] at (800,280){\tiny 26};

\node[font=\bf,color=colorL2] at (100,350){\tiny 33};
\node[font=\bf,color=colorL2] at (200,360){\tiny 34};
\node[font=\bf,color=colorL2] at (300,370){\tiny 35};
\node[font=\bf,color=colorL2] at (400,370){\tiny 35};
\node[font=\bf,color=colorL2] at (500,370){\tiny 35};
\node[font=\bf,color=colorL2] at (600,370){\tiny 35};
\node[font=\bf,color=colorL2] at (700,370){\tiny 35};
\node[font=\bf,color=colorL2] at (800,380){\tiny 36};

\node[font=\bf,color=colorLN] at (800,130){\tiny(3/8)};
\node[font=\bf,color=colorLN] at (700,130){\tiny(3/7)};
\node[font=\bf,color=colorLN] at (600,110){\tiny(2/7)};
\node[font=\bf,color=colorLN] at (500,110){\tiny(2/6)};
\node[font=\bf,color=colorLN] at (400,100){\tiny(1/7)};
\node[font=\bf,color=colorLN] at (300,40){\tiny(1/5)};
\node[font=\bf,color=colorLN] at (200,70){\tiny(1/4)};
\node[font=\bf,color=colorLN] at (100,70){\tiny(1/4)};

\node[font=\bf,color=colorNewton] at (400,60){\tiny 8};
\node[font=\bf,color=colorNewton] at (300,90){\tiny 7};
\node[font=\bf,color=colorNewton] at (200,30){\tiny 5};
\node[font=\bf,color=colorNewton] at (100,30){\tiny 5};
\end{axis} 

\end{tikzpicture}
}}
\hspace{0.02\textwidth}
   \subfloat[][Computational time in seconds.\label{fig: CPU time ex1}]{\resizebox{0.48\textwidth}{!}{
\begin{tikzpicture}[thick, scale=1]
\tikzstyle{every node}=[font=\small]
\begin{axis}
[ legend entries={{\scriptsize $L_1$,\scriptsize $L_2$,\scriptsize Newton,\scriptsize L/N}},
    legend style={at={(0.0,0.8)},anchor=west},
    xlabel={\scriptsize Time},
    xmin = 20,
    xmax = 85,
    ymin = 0,
    ymax = 1800,
xlabel={$\sqrt{2}/h$}, ylabel={CPU time [s]} ]
\addplot[color=colorL1,mark=triangle,mark options={scale=1.5}] 
coordinates { (30,332) (40,442) (50,577) (60,772) (70,991) (80,1230) }; 
\addplot[color=colorL2,mark=square,mark options={scale=1.5}] 
coordinates {  (30,469) (40,599) (50,832) (60,1032) (70,1373) (80,1666) }; 
\addplot[color=colorNewton,mark=halfcircle*,mark options={scale=1.5}] 
coordinates { (30,104) (40,151) }; 
\addplot[color=colorLN,mark=o,mark options={scale=1.5}] 
coordinates { (30,91) (40,153) (50,193) (60,275) (70,373) (80,518) };

\node[font=\bf,color=colorL1] at (100,28){\tiny 332};
\node[font=\bf,color=colorL1] at (200,38){\tiny 442};
\node[font=\bf,color=colorL1] at (300,64){\tiny 577};
\node[font=\bf,color=colorL1] at (390,84){\tiny 772};
\node[font=\bf,color=colorL1] at (490,107){\tiny 991};
\node[font=\bf,color=colorL1] at (600,130){\tiny 1230};

\node[font=\bf,color=colorL2] at (100,54){\tiny 469};
\node[font=\bf,color=colorL2] at (200,67){\tiny 599};
\node[font=\bf,color=colorL2] at (300,91){\tiny 832};
\node[font=\bf,color=colorL2] at (385,110){\tiny 1032};
\node[font=\bf,color=colorL2] at (485,145){\tiny 1373};
\node[font=\bf,color=colorL2] at (600,156){\tiny 1666};

\node[font=\bf,color=colorLN] at (600,59){\tiny 518};
\node[font=\bf,color=colorLN] at (500,44){\tiny 373};
\node[font=\bf,color=colorLN] at (400,34){\tiny 275};
\node[font=\bf,color=colorLN] at (300,25){\tiny 193};
\node[font=\bf,color=colorLN] at (200,9){\tiny 153};
\node[font=\bf,color=colorLN] at (100,4){\tiny 91};

\node[font=\bf,color=colorNewton] at (100,17){\tiny 104};
\node[font=\bf,color=colorNewton] at (200,23){\tiny 151};

\end{axis} 
\end{tikzpicture}
    }}
\caption{Test case 1: Strictly unsaturated medium. Performance metrics  for all linearization schemes for fixed $\tau=0.01$ and varying mesh size.}
\end{figure}

\usetikzlibrary{patterns}
\makeatletter
\tikzset{nomorepostaction/.code=\let\tikz@postactions\pgfutil@empty}
\makeatother
\pgfplotsset{yticklabel style={text width=3em,align=right}}
\begin{figure}[ht!]
    \centering
    \subfloat[][Number of iterations for different time step sizes.\label{fig: number of itr ex1}]{\resizebox{0.5\textwidth}{!}{
        \begin{tikzpicture}
    \begin{axis}[scale only axis,
          grid=major,
          height=7cm,
          width=8cm, x tick label style={ /pgf/number format/1000 sep=}, ylabel=Number of iterations, enlargelimits=0.15, legend style={at={(0.5,-0.2)}, anchor=north,legend columns=-1}, ybar, bar width=10pt,xtick={1930,1940,1950,1960},
    xticklabels={$0.001$,$0.01$,$0.1$,$1$},xlabel=$\tau$,ytick={0,10,20,30,40,55,84,114} ] 
\addplot[style= {fill=colorLN, mark=none, postaction={pattern=vertical lines}}] coordinates {(1930,5) (1940,8) (1950,11) (1960,8) }; 
\addplot[style= {fill=colorNewton, mark=none, postaction={pattern=horizontal lines}}] coordinates {(1930,5) (1940,8)  }; 
\addplot[style= {fill=colorL1, mark=none, postaction={pattern=north east lines}}] coordinates {(1930,40) (1940,25) (1950,29) (1960,84) }; 
\addplot[style= {fill=colorL2, mark=none, postaction={pattern=dots}}] coordinates {(1930,55) (1940,34) (1950,40) (1960,114)};
\legend{L/N,Newton,$L_1$,$L_2$} 

\node[color=colorLN] at (26.75,8) {\tiny(1/7)};
\node[color=colorLN] at (26.75,8) {\tiny(1/7)};
\node[color=colorLN] at (26.75,8) {\tiny(1/7)};
\node[color=colorLN] at (16.75,11) {\tiny(3/8)};
\node[color=colorLN] at (16.75,11) {\tiny(3/8)};
\node[color=colorLN] at (16.75,11) {\tiny(3/8)};
\node[color=colorLN] at (6.75,8) {\tiny(1/7)};
\node[color=colorLN] at (6.75,8) {\tiny(1/7)};
\node[color=colorLN] at (6.75,8) {\tiny(1/7)};
\node[color=colorLN] at (-3,7) {\tiny(1/4)};
\node[color=colorLN] at (-3,7) {\tiny(1/4)};
\node[color=colorLN] at (-3,7) {\tiny(1/4)};
\end{axis} 
\end{tikzpicture}
    }}
    \subfloat[][Total computational time in seconds for different time step sizes.\label{fig: computation time ex1}]{\resizebox{0.5\textwidth}{!}{
         \begin{tikzpicture}
    \begin{axis}[scale only axis,
          grid=major,
          height=7cm,
          width=8cm, x tick label style={ /pgf/number format/1000 sep=}, ylabel={CPU time [s]}, enlargelimits=0.15, legend style={at={(0.5,-0.2)}, anchor=north,legend columns=-1}, ybar, bar width=10pt,xtick={1930,1940,1950,1960},
    xticklabels={$0.001$,$0.01$,$0.1$,$1$},xlabel=$\tau$ ] 
\addplot[style = {fill=colorLN, mark=none, postaction={pattern=horizontal lines}}] coordinates {(1930,101) (1940,153) (1950,189) (1960,154) }; 
\addplot[style = {fill=colorNewton, mark=none, postaction={pattern=vertical lines}}] coordinates {(1930,101) (1940,151)  }; 
\addplot[style= {fill=colorL1, mark=none, postaction={pattern=north east lines}}]  coordinates {(1930,698) (1940,442) (1950,511) (1960,1451) }; 
\addplot[style= {fill=colorL2, mark=none, postaction={pattern=dots}}] coordinates {(1930,958) (1940,599) (1950,699) (1960,1962)};
\legend{L/N,Newton,$L_1$,$L_2$} 
\end{axis} 
\end{tikzpicture}
    }}
    \caption{Test case 1: Strictly unsaturated medium: Performance comparison for all of the linearization schemes for different time step sizes and fixed mesh size $h=\sqrt{2}/40$.}
    \label{fig: convergence time step ex 1}
    \end{figure}
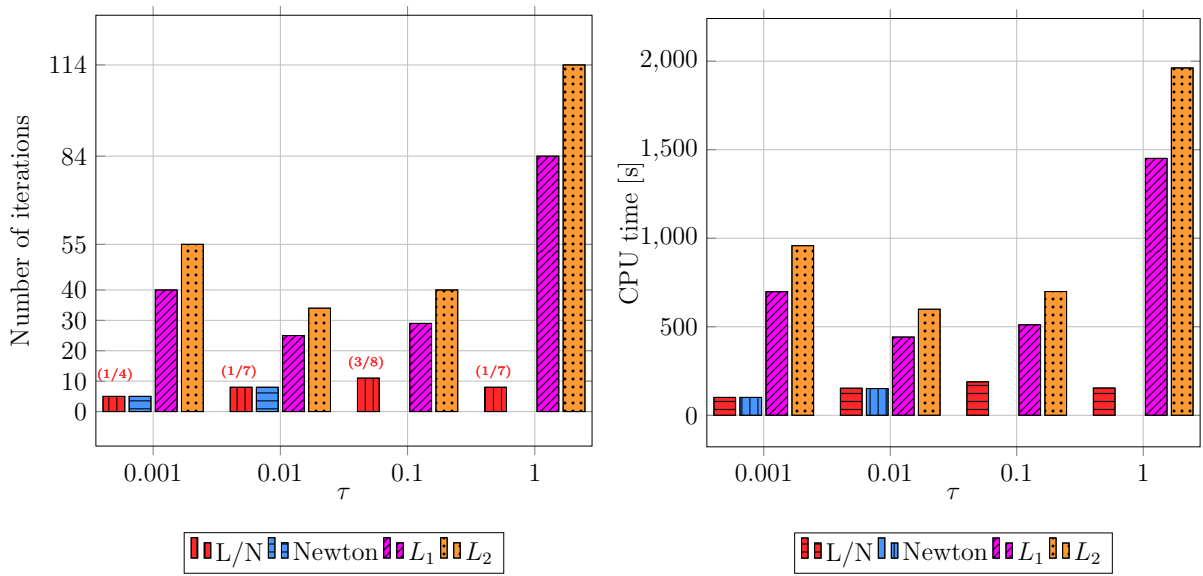

Then, the performance of the linearization schemes is compared in terms of computational time, cf.\ \Cref{fig: CPU time ex1} and \Cref{fig: computation time ex1}. 
One can observe virtually the same performance for the hybrid method as for Newton's method when the latter converges. The former in fact is sometimes slightly faster, due to each L-scheme iteration being slightly less expensive than a Newton iteration, see Remark~\ref{rem:newton_slow}. In addition, the hybrid method continues to show the same performance for the cases in which Newton's method does not converge. Finally, \Cref{fig: CPU time ex1} shows that, for all meshes, the computational time of the L-schemes is consistent with the reported numbers of iterations in \Cref{fig: convergence ex1} with $L_1$ being the fastest. Although it uses more than double the computational time of the L/N-scheme.

Overall, the newly proposed L/N-scheme shows the best performance. It is as fast as Newton's method when it converges, and is significantly more robust. 


\begin{remark}[Computational time per iteration]\label{rem:newton_slow}
 It is known that condition numbers for matrices coming from systems linearized by Newton's method are higher than for those linearized by the L-scheme \cite{ListFlorian2016Asoi}. Therefore, each iteration of Newton's method, when implemented without preconditioning, takes more time than each L-scheme iteration. 
 \end{remark}

 \begin{remark}[Computational time for the coarsest mesh]
     The computational times of the coarsest meshes are omitted due to the use of multiprocessing in the implementations. This causes the most time consuming part to be the spawn process of the local assembly on each element. As a result,  the computational times for the coarsest meshes are very similar for all the linearization methods.
 \end{remark}



\subsubsection{Switching characteristics}

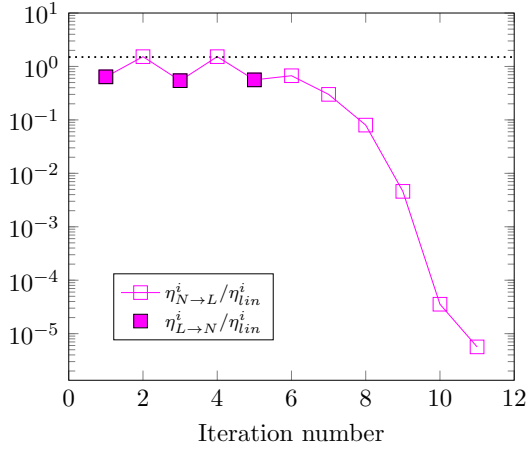
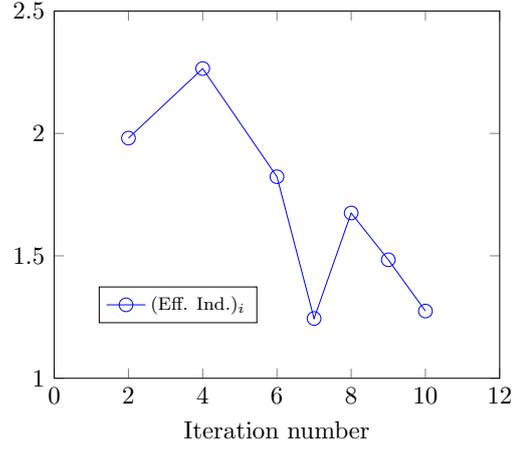
\begin{figure}
\centering

\tikzstyle{every node}=[font=\small]
    \subfloat[][Evolution of switching indicators for L/N-scheme where the dashed line is $C_{\rm tol}=1.5$. The L/N-scheme oscillates between the linearization strategies, but eventually recovers. ]{\resizebox{0.48\textwidth}{!}{
    \begin{tikzpicture} 
\begin{semilogyaxis}
[ legend entries={{\scriptsize $\eta_{N\to L}^i/\eta_{lin}^i$,\scriptsize $\eta_{L\to N}^i/\eta_{lin}^i$}},
    ,
    xlabel={\scriptsize Time},
    xmin = 0,
    xmax = 12,
    ymax = 10,
xlabel=Iteration number,legend style={at={(0.1,0.2)},anchor=west}] 

\addplot[color=colorL1,mark=square,mark options={scale=1.5}] 
coordinates { (1,0.6392) (2,1.52558603) (3,0.54187) (4,1.52558603) (5,0.561055)   (6,0.67068999) (7,0.30026283) (8,0.07990169) (9,0.00459914) (10,3.5533441e-05) (11,5.65431635e-06)  }; 
\addplot[mark=square*,only marks,fill=colorL1,mark options={scale=1.5}] coordinates {(1,0.6392)  (3,0.54187)  (5,0.561055) };

%
\addplot[thick,dotted] coordinates {(0,1.5) (12,1.5)};

\end{semilogyaxis} 
\end{tikzpicture}
}}\hspace{0.02\textwidth}
    \subfloat[][Efficency index. Notice that the iterations correspond to the ones in Figure~\ref{fig:switching indictor ex1}(a), and that only the ones where the Newton method is performed are counted, i.e., iteration 1,3 and 5 are removed.]{\resizebox{0.48\textwidth}{!}{
    \begin{tikzpicture} 
\begin{axis}
[ legend entries={{\scriptsize (Eff. Ind.)$_i$}},
    ,
    xlabel={\scriptsize Time},
    xmin = 0,
    xmax = 12,
    ymin = 1,
    ymax = 2.5,
xlabel=Iteration number,legend style={at={(0.1,0.2)},anchor=west}] 

\addplot[mark=o,color=blue,mark options={scale=1.5}] coordinates {(2,1/0.5048229088168802)  (4,1/0.4414872798434442) (6,1/0.5484988122403807) (7,1/0.803815841865507) (8,1/0.5970503891847604) (9,1/0.6739188704302732) (10,1/ 0.7845605923627317)   };%

\end{axis}
\end{tikzpicture}    }}

\caption{Test case 1: Strictly unsaturated medium. Evolution of switching indicators for the L/N-scheme and efficiency indices \eqref{eq: eff ind} for the Newton iterations (see \Cref{rem: effectivity}). Here, the mesh size is $h=\sqrt{2}/80$ and time step size $\tau=0.01$.}
\label{fig:switching indictor ex1}
\end{figure}

Finally, the dynamic switch between the L-scheme and Newton's method is inspected in further detail. In~\Cref{fig:switching indictor ex1}, the evolution of the indicators for the switch is displayed for a fixed mesh and time step size. The example particularly demonstrates the ability of the hybrid method to switch back and forth between both linearizations before  switching fully to Newton. In addition, the final number of L-scheme iterations is kept at its minimum. The plot also shows the effectivity indices introduced in \eqref{eq: eff ind} and discussed in \Cref{rem: effectivity}. The effectivity index is greater than 1 in all cases, which validates \Cref{prop:LtoN,prop:NtoN} and it stays between 1.27 to 2.3, implying that the estimators $\eta_{\LN}^i$ and $\eta^i_{\NL}$ are sharp.



\subsection{Test case 2: Variably saturated medium}

The example parameters are as in \Cref{table: example parameters}, Test case 3. We consider a variably saturated medium, 
 $\Omega = \Omega_{gw}\cup\Omega_{vad}$, where the groundwater zone is $\Omega_{gw}=[0,1]\times[0,1/4)$ and a vadoze zone is $\Omega_{vad}=[0,1]\times[1/4,1]$.
Here, we consider the time interval $[0,T]$, where $T=0.01$ and we only take one time step with $\tau=0.01$.  As initial condition, we choose the pressure head
$$\boldsymbol{\psi}^0(x,z)=\begin{cases}-z+1/4 &(x,z)\in  \Om_{gw}\\ -3 &(x,z)\in  \Om_{vad}, \end{cases}$$ where $x$ represents the positional variable in the horizontal direction and $z$ in the vertical direction. 
 On the surface a constant  Dirichlet boundary condition is imposed, being equal to the initial condition at all times.  For the rest of the boundary no-flow boundary conditions are used. We apply the following source term 
 $$
f(x,z)=\begin{cases} 0 &(x,z)\in \Om_{gw}\\ 0.006\cos\left(\frac{4}{3}\pi(z-1)\right)\sin\left(2\pi x\right)
&(x,z)\in \Om_{vad}.\end{cases}$$
After one time step the pressure head profile is given in \Cref{fig:saturatedpressure}.
\begin{figure}[H]
    \centering
    \includegraphics[width=8cm]{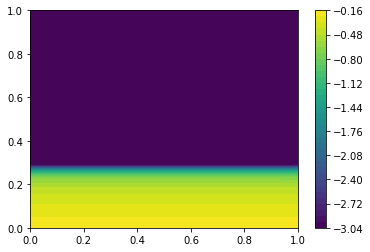}
    \caption{Test case 2: Variably saturated medium: Pressure head profile at $T=0.01$.}
    \label{fig:saturatedpressure}
\end{figure}

\begin{figure}[H]

\centering
    \subfloat[][Total number of iterations. The numbers in the red parentheses correspond to (number of L-scheme iterations/number of Newton iterations).\label{fig: convergence ex2}]{\resizebox{0.5\textwidth}{!}{
\begin{tikzpicture}[thick, scale=1]
\tikzstyle{every node}=[font=\small]
\begin{axis}
[
    xmin = 0,
    xmax = 85,
    ymin = 0,
    ymax = 70,
xlabel={$\sqrt{2}/h$}, ylabel= Number of iterations,ylabel style={at={(-0.1,0.5)}}] 
\addplot[color=colorL1,mark=triangle,mark options={scale=1.5}] 
coordinates { (10,43) (20,43) (30,35) (40,35) (50,36) (60,36) (70,38) (80,39) }; 
\addplot[color=colorL2,mark=square,mark options={scale=1.5}] 
coordinates { (10,67) (20,66) (30,51) (40,54) (50,56) (60,55) (70,57) (80,58) }; 
\addplot[color=colorLN,mark=o,mark options={scale=1.5}] 
coordinates { (10,43) (20,9) (30,12) (40,11) (50,11) (60,10) (70,10) (80,10) };

\node[font=\bf,color=colorLN] at (800,130){\tiny(4/6)};
\node[font=\bf,color=colorLN] at (700,130){\tiny(5/5)};
\node[font=\bf,color=colorLN] at (600,130){\tiny(5/5)};
\node[font=\bf,color=colorLN] at (500,140){\tiny(8/3)};
\node[font=\bf,color=colorLN] at (400,130){\tiny(6/3)};
\node[font=\bf,color=colorLN] at (300,150){\tiny(8/4)};
\node[font=\bf,color=colorLN] at (200,60){\tiny(5/4)};
\node[font=\bf,color=colorLN] at (100,460){\tiny(43)};

\node[font=\bf,color=colorL1] at (800,420){\tiny(39)};
\node[font=\bf,color=colorL1] at (700,410){\tiny(38)};
\node[font=\bf,color=colorL1] at (600,390){\tiny(36)};
\node[font=\bf,color=colorL1] at (500,390){\tiny(36)};
\node[font=\bf,color=colorL1] at (400,380){\tiny(35)};
\node[font=\bf,color=colorL1] at (300,390){\tiny(35)};
\node[font=\bf,color=colorL1] at (200,460){\tiny(43)};
\node[font=\bf,color=colorL1] at (100,490){\tiny(43)};

\node[font=\bf,color=colorL2] at (800,610){\tiny(58)};
\node[font=\bf,color=colorL2] at (700,600){\tiny(57)};
\node[font=\bf,color=colorL2] at (600,580){\tiny(55)};
\node[font=\bf,color=colorL2] at (500,590){\tiny(56)};
\node[font=\bf,color=colorL2] at (400,570){\tiny(54)};
\node[font=\bf,color=colorL2] at (300,480){\tiny(51)};
\node[font=\bf,color=colorL2] at (200,620){\tiny(66)};
\node[font=\bf,color=colorL2] at (100,630){\tiny(67)};

\end{axis} 

\end{tikzpicture}
}}
   \subfloat[][Computational time in seconds.\label{fig: cpu ex2}]{\resizebox{0.5\textwidth}{!}{
\begin{tikzpicture}[thick, scale=1]
\tikzstyle{every node}=[font=\small]
\begin{axis}
[legend entries={{\scriptsize $L_1$,\scriptsize $L_2$,\scriptsize L/N}},
    xlabel={\scriptsize Time},
    xmin = 20,
    xmax = 85,
    ymin = 0,
    ymax = 2800,
xlabel={$\sqrt{2}/h$}, ylabel= {CPU time [s]},ylabel style={at={(-0.15,0.5)}},legend style={at={(0.05,0.8)},anchor=west}] 
\addplot[color=colorL1,mark=triangle,mark options={scale=1.5}] 
coordinates {  (30,477) (40,616) (50,828) (60,1058) (70,1442) (80,1808) }; 
\addplot[color=colorL2,mark=square,mark options={scale=1.5}] 
coordinates { (30,700) (40,939) (50,1269) (60,1614) (70,2141) (80,2676) }; 
\addplot[color=colorLN,mark=o,mark options={scale=1.5}] 
coordinates {  (30,169) (40,173) (50,273) (60,304) (70,395) (80,484) };

\node[font=\bf,color=colorLN] at (600,60){\tiny(484)};
\node[font=\bf,color=colorLN] at (500,50){\tiny(395)};
\node[font=\bf,color=colorLN] at (400,42){\tiny(304)};
\node[font=\bf,color=colorLN] at (300,39){\tiny(273)};
\node[font=\bf,color=colorLN] at (200,28){\tiny(173)};
\node[font=\bf,color=colorLN] at (100,28){\tiny(169)};

\node[font=\bf,color=colorL1] at (600,192){\tiny(1808)};
\node[font=\bf,color=colorL1] at (485,156){\tiny(1442)};
\node[font=\bf,color=colorL1] at (400,92){\tiny(1058)};
\node[font=\bf,color=colorL1] at (300,70){\tiny(828)};
\node[font=\bf,color=colorL1] at (200,50){\tiny(616)};
\node[font=\bf,color=colorL1] at (100,39){\tiny(477)};

\node[font=\bf,color=colorL2] at (608,250){\tiny(2676)};
\node[font=\bf,color=colorL2] at (485,228){\tiny(2141)};
\node[font=\bf,color=colorL2] at 
(385,173){\tiny(1614)};
\node[font=\bf,color=colorL2] at 
(297,141){\tiny(1269)};
\node[font=\bf,color=colorL2] at 
(200,107){\tiny(939)};
\node[font=\bf,color=colorL2] at 
(100,83){\tiny(700)};
\end{axis} 

\end{tikzpicture}
   }}
\caption{Test case 2: Variably saturated medium: Performance metrics for all linearization schemes for fixed $\tau=0.01$ and varying mesh size.}
\end{figure}
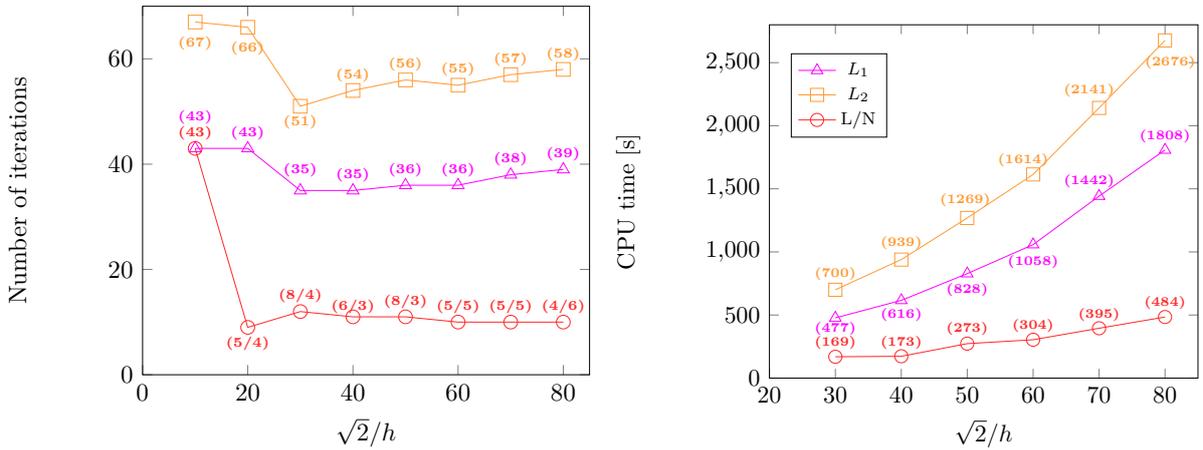

\subsubsection{Comparison of convergence properties.}
The iteration count for the second test case for different mesh sizes and fixed time step for all linearization schemes is illustrated in \Cref{fig: convergence ex2}.  Again the L-scheme converges in every case. However, Newton’s method does not converge for any mesh size. The hybrid method needs the fewest number of iterations, which shows that the dynamic switch is successful.

The CPU time performance of the linearization schemes is compared in \Cref{fig: cpu ex2}. Both versions of the L-scheme takes computational times consistent with the number of iterations, with the simulations with the parameter $L_1$ being less expensive. However, the L-scheme (using $L_1$) requires approximately 373\% of the computational time of the hybrid method including the computation of the switching indicators. In addition, the benefit of a few additional L-scheme iterations further decreases the computational time of the hybrid method.  

\subsubsection{Switching characteristics}
We also give a more in-depth look to the dynamic switch between the Newton's method and the L-scheme. In \Cref{fig:switching indictor ex2}, the evolution of the switching indicators is shown for a fixed  time step and a fixed mesh size. After 8 L-scheme iterations the switching indicator $\eta_{\LN}$ becomes lower than $C_{\rm tol}$ and then Newton's method converges. From \Cref{fig: convergence ex2} the number of L-scheme iterations required before the switching indicator becomes small enough to switch to Newton's method varies with the mesh size. 
Note that for the coarsest mesh no switch to Newton's method happens.

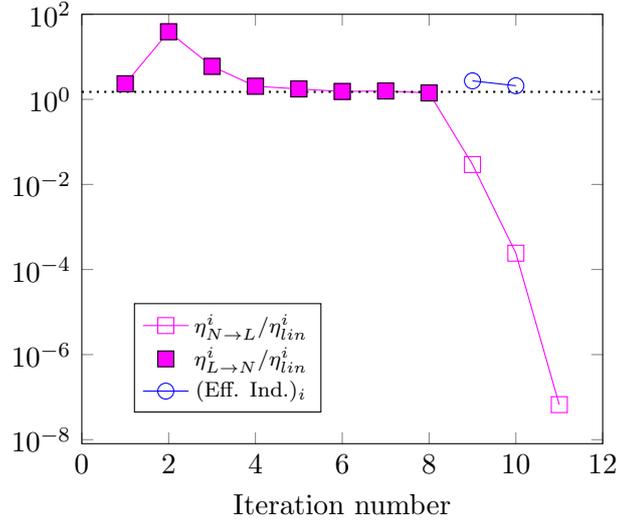
\begin{figure}[H]
\centering
\begin{tikzpicture} 
\tikzstyle{every node}=[font=\small]
\begin{semilogyaxis}
[ legend entries={{\scriptsize $\eta_{N\to L}^i/\eta_{lin}^i$,\scriptsize $\eta_{L\to N}^i/\eta_{lin}^i$,\scriptsize (Eff. Ind.)$_i$}},
    ,
    xlabel={\scriptsize Time},
    xmin = 0,
    xmax = 12,
    ymax = 100,
xlabel=Iteration number,legend style={at={(0.1,0.2)},anchor=west}] 

\addplot[color=colorL1,mark=square,mark options={scale=1.5}] 
coordinates { (1,2.32277169) (2,38.84575595) (3,5.96631624) (4,2.03543675) (5,1.75522805)   (6,1.53010559) (7,1.57266833) (8,1.4130238) (9,0.02936379) (10,0.00024202) (11,6.65890813e-08)  }; 
\addplot[mark=square*,only marks,fill=colorL1,mark options={scale=1.5}] coordinates {(1,2.32277169) (2,38.84575595) (3,5.96631624) (4,2.03543675) (5,1.75522805)   (6,1.53010559) (7,1.57266833) (8,1.4130238) };

\addplot[,color=blue,mark=o,mark options={scale=1.5}] coordinates {(9,1/0.36724817417338246) (10,1/0.4785528002562701)   };
\addplot[thick,dotted] coordinates {(0,1.5) (12,1.5)};

\end{semilogyaxis} 

\end{tikzpicture}
\caption{Test case 2: Variably saturated medium: Evolution of switching indicators for L/N-scheme for fixed $h=\sqrt{2}/50$ and $\tau=0.01$. The dashed line is $C_{\rm tol}=1.5$, the switching criterion from L-scheme to Newton's method. The effectivity indices \eqref{eq: eff ind} corresponding to the Newton iterations are also plotted and they remain below 2.8.}
\label{fig:switching indictor ex2}
\end{figure}

\subsection{Test case 3: Benchmark problem}

Here, we consider a known benchmark problem \cite{schneid2000hybrid}, also used e.g. in \cite{ListFlorian2016Asoi}, which models the recharge of a groundwater reservoir from a drainage trench in two spatial dimensions. The domain $\Om\subset\mathbb{R}^2$ represents a vertical segment of the subsurface. One portion of the right side of the domain is fixed by a constant Dirichlet boundary condition. A time-dependent Dirichlet boundary condition on parts of the upper boundary is used to mimic the drainage trench. No-flow conditions are utilized on the remaining parts of the boundary. The used parameters are given in \Cref{table: example parameters} Test case 3, corresponding to silt loam.
The geometry is given by
\begin{align*}
    \Om &= [0,2]\times[0,3],\\
    \Gamma_{D_1}&=[0,1]\times(3),\\
    \Gamma_{D_2}&=(2)\times[0,1],\\
    \Gamma_N&=\Omega\backslash\left\{\Gamma_{D_1}\cup\Gamma_{D_2}\right\},\\
\end{align*}
and the initial pressure head distribution and boundary conditions are
\begin{align*}
        &\psi(0,x,z)=1-z\\
    	&\psi(t,x,z) = \begin{cases}
			-2+35.2t, \quad &\mbox{if } t\leq \frac{1}{16},\quad \mbox{on } \Gamma_{D_1},\\
			0.2,  \quad &\mbox{if } t> \frac{1}{16},\quad \mbox{on } \Gamma_{D_1},\\
			1-z,\quad &\mbox{on } \Gamma_{D_2},\\
		\end{cases}\\
		&-K(\theta(\psi(t,x,z)))\nabla(\psi(t,x,z) +z)\cdot\boldsymbol{\nu}=0,\quad \mbox{on }\Gamma_N,\\
\end{align*}
where $\boldsymbol{\nu}$ is the outward normal vector.
 The solution is computed over 9 timesteps, where the time unit is in days, with time step size $\tau=1/48$ and with a regular mesh consisting of 2501 nodes. The pressure head profile at the final time for the L/N-scheme is shown in \Cref{fig: benchmark}.

\begin{figure}[H]
    \centering
    \includegraphics[width=6cm]{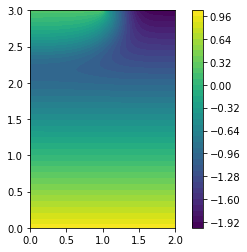}
    \caption{Test case 3: Benchmark problem: Pressure head profile at $4.5$ hours.}
    \label{fig: benchmark}
\end{figure}

\begin{center}
\begin{table}[H]
\setlength\tabcolsep{0pt} 
\centering
\begin{tabular*}{0.6\textwidth}{@{\extracolsep{\fill}} c *{2}{d{2}} }
\toprule
 
 & \multicolumn{1}{c}{No. Itr} & \multicolumn{1}{c}{CPU time [s]} \\ 
\midrule
$L_1$     & 274 & 6136 \\

$L_2$     & 330 & 7356  \\
Newton
              & 39 & 980   \\
L/N      & (10/30) & 1021\\
\bottomrule
\end{tabular*}
\caption{Test case 3: Benchmark problem: Performance metrics for 2501 nodes.}
\label{table: ex 3}
\end{table}
\end{center}

\subsubsection{Comparison of convergence properties.}

The performance of all schemes for test case 3 is displayed in \Cref{table: ex 3}. All schemes converge for this example. The Newton method requires the least amount of iterations. However, the hybrid method only needs one more iteration. Both uses  significantly less iterations than the L-schemes. For all time steps except one, only one L-scheme iteration is needed per time step, which indicates a successful dynamic switch for almost all time steps.

The computational time for the L-schemes is much higher than both Newton's method and the hybrid method, which is consistent with the expense per iteration discussed in \Cref{rem:newton_slow}. More significantly, the L/N-scheme performs almost the same as Newton's method.

\section{Conclusions}\label{sec: conclusions}

In this paper, we considered solving Richards' equation, which models the flow of water through saturated/\-unsaturated porous media (soil). After applying backward Euler time-discretization and continuous Galerkin finite element space-discretization to Richards' equation, to solve the resulting nonlinear finite-dimensional problem  we developed a hybrid iterative linearization strategy that combines the L-scheme with the Newton method. The idea behind this is to use the robust, but only first-order convergent L-scheme to stabilize the quadratically convergent Newton method. The switching between the two schemes is done in an adaptive manner using {\it a posteriori} indicators which predict the linearization error of the next iteration using a concept of iteration-dependent energy norms. After each iteration, it is checked whether the Newton method is predicted to decrease the linearization error of the next iteration. If so, then the Newton method is used, otherwise, the iteration is done using the L-scheme.
The hybrid scheme is now robust, but still quadratically convergent after switching to the Newton scheme. 

The performance of the hybrid scheme is tested on illustrative, realistic numerical examples which reveal that the scheme is as robust as the L-scheme and it converges in cases where Newton fails. Moreover, in cases when Newton converges, the hybrid scheme takes roughly the same amount of iterations and computational time and is considerably faster than even the optimized L-scheme. Lastly, we comment that the scheme is quite general as it can, in principle, be extended to other spatial discretization and linearization methods.




\begin{appendices}

\section{An adaptive L-scheme}\label{sec: LtoL}

As discussed in \Cref{sec: intro,sec: Lin Scheme}, the L-scheme converges unconditionally provided that $L\geq \frac{1}{2}\sup_{\xi\in \R}\theta'(\xi)$ and the time step size $\t$ is smaller than a constant independent of the mesh size. However, numerical results in \cite{ListFlorian2016Asoi} suggest that the optimal rate of convergence of the L-scheme is obtained for a considerably smaller $L$ although convergence cannot always be guaranteed for such values. Hence, to speed up the computations, it is possible to start the iterations with a smaller value of $L$ and then use the {\it a posteriori} estimates to decide if $L$ is to be increased or not. Analogous to \Cref{prop:LtoN,prop:NtoN} we state a result that allows us to do this rigorously.

\begin{proposition}[Error control of L-scheme]\label{prop:LtoL}
For a given $\psi^{n,0}_h,\, \psi^{n-1}_h\in V_h$, let $\{\psi^{n,j}_h\}^{i+1}_{j=1}\subset V_h$ solve \eqref{eq: L-scheme} for some $i\in \mathbb{N}$. Then under Assumption~\ref{as:auxiliary_functions},
$$
\norm{\psi^{n,i+1}_h-\psi^{n,i}_h}_{L,\psi^{n,i}_h}\leq \eta_{\LL}^i,
$$
where $$\eta_{\LL}^i :=\left([\eta_{\LL}^{i,{\rm poten}}]^2 + \t [\eta_{\LL}^{i,{\rm flux}}]^2\right)^{\frac{1}{2}}$$
with \begin{align*}
    &\eta_{\LL}^{i,{\rm poten}}:= \|L^{-\frac{1}{2}}(L(\psi^{n,i}_h-\psi^{n,i-1}_h)-(\theta(\psi^{n,i}_h)-\theta(\psi^{n,i-1}_h)))\|,\\
    &\eta_{\LL}^{i,{\rm flux}}:=  \left\|
         (K(\theta(\psi_h^{n,i}))-K(\theta(\psi_h^{n,i-1}))) K(\theta(\psi_h^{n,i}))^{-\frac{1}{2}}\del (\psi^{n,i}_h+z)\right\| .
\end{align*}
\end{proposition}

The detailed proof is again omitted. Observe that for the estimate above, neither Assumption \ref{as:convection} nor any separate treatment of the degenerate domains is required. 

\subsection{L-adaptive algorithm}
Based on \Cref{prop:LtoL}, we propose an 
algorithm that  selects optimal $L$-values adaptively.

\begin{algorithm}[H]
	\caption{The $L$-adaptive scheme}\label{alg: L-adaptive}
	\begin{algorithmic}
		\Require $\boldsymbol{\psi}^{n,0}\in L^2(\Om)$ as initial guess,  $L_M:=\sup_{\psi\in \R}\theta'(\psi)$, and $L_m:=L_M/8$
  
        \Ensure $C_{\LL}=\sqrt{2}$, $L=L_m$
	\For{i=1,2,..}
  \State
   Compute iterate using  L-scheme, i.e., \eqref{eq: L-scheme}
        \If{$\eta^i_{\LL}>\eta^i_{\rm lin}$} 
        
          \State Replace $L_m=L$,   $L=\min(C_{\LL} L,L_M)$, and \textbf{continue}.
	\ElsIf{$\eta_{\LL}^j>0.8\, \eta^j_{\rm lin}$ for $j\in \{i,i-1,i-2\}$}
		  \State Replace $L=\max(0.9 L,1.1L_m)$ and \textbf{continue}.
		
  \EndIf

		\EndFor

	\end{algorithmic}
\end{algorithm}

\subsection{Numerical result}
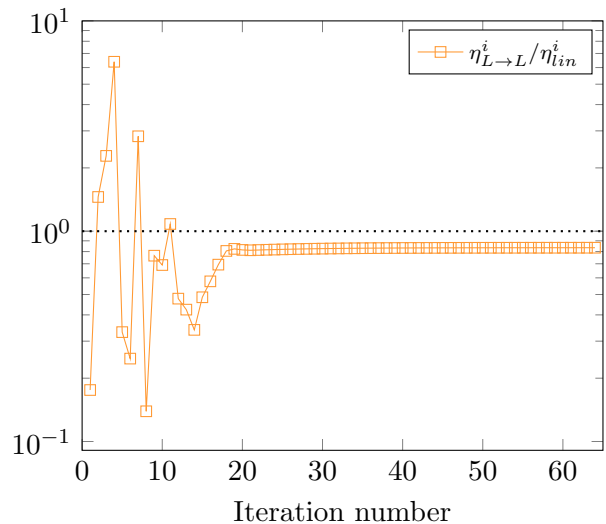
\begin{figure}[H]
\centering
\begin{tikzpicture} 
\tikzstyle{every node}=[font=\small]
\begin{semilogyaxis}
[ legend entries={{\scriptsize $\eta_{L\to L}^i/\eta_{lin}^i$}},
    ,
    xlabel={\scriptsize Time},
    xmin = 0,
    xmax = 65,
    ymax = 10,
xlabel=Iteration number] 

\addplot[color=colorL2,mark=square,mark options={scale=1}] 
coordinates { (1,0.175721610581879)
(2,1.45460629244473)
(3,2.28101218041231)
(4,6.39219157005095)
(5,0.331235652168848)
(6,0.247852830535918)
(7,2.82702436154929)
(8,0.139224515533529)
(9,0.764804883515017)
(10,0.690524321177711)
(11,1.08203227005026)
(12,0.477903544849466)
(13,0.423715095636787)
(14,0.339488557479197)
(15,0.485098894879873)
(16,0.577562802134723)
(17,0.693418663695379)
(18,0.805636486091624)
(19,0.823988406216277)
(20,0.814606614010852)
(21,0.811519129540219)
(22,0.813793682860139)
(23,0.815838346281356)
(24,0.817731015055316)
(25,0.819580976496173)
(26,0.821172796476119)
(27,0.822684045347442)
(28,0.823994748122771)
(29,0.825216109270237)
(30,0.826268288970343)
(31,0.827241729304255)
(32,0.828079508329717)
(33,0.828849030755492)
(34,0.829510617841569)
(35,0.830115215072827)
(36,0.830634608875146)
(37,0.831107417333625)
(38,0.831513455463428)
(39,0.831882056673224)
(40,0.832198620194322)
(41,0.832485491619408)
(42,0.832731981560047)
(43,0.832955149539867)
(44,0.833147071270109)
(45,0.833320799907759)
(46,0.833470388521895)
(47,0.833605848017136)
(48,0.833722664580352)
(49,0.833828534915165)
(50,0.83391999856077)
(51,0.834002988514966)
(52,0.834074829025534)
(53,0.834140106124171)
(54,0.834196736154638)
(55,0.834248273838805)
(56,0.834293087303316)
(57,0.834333939521634)
(58,0.834369546671184)
(59,0.834402062380314)
(60,0.834430472891768)
(61,0.834456461283321)
(62,0.834479225186496)
(63,0.834500083312994)
(64,0.834518399120227)
}; 
\addplot[thick,dotted] coordinates {(0,1) (65,1)};

\end{semilogyaxis} 

\end{tikzpicture}
\caption{Test case 1: Strictly unsaturated medium: L-scheme with L-adaptivity and initial stabilization parameter $L_0=L_2/8$, $h=\sqrt{2}/40$ and $\t=1$.}
\label{fig:L-adaptivity}
\end{figure}

In \Cref{fig:L-adaptivity} we show a result where the $L$-adaptive scheme is superior to a fixed $L$-approach. In this case, $L_\theta/2$ is too small for convergence due to a large time step size. Compared with fixed $L_1$ with the same mesh size and time step size, see \Cref{fig: convergence time step ex 1}, the number of iterations is improved by 20. For smaller time steps, the numerical results reveal that Algorithm \ref{alg: L-adaptive}  results in  roughly the same number of iterations compared to a fixed and optimized $L=L_1$  lesser than $L_\theta$. But in all examples considered, it uses fewer iterations than simply choosing $L=L_2=L_\theta$. The advantage of such an adaptive technique is that an optimization study of $L$ does not need to be conducted prior to the simulation. However, since the $L$-adaptive strategy does not significantly improve the behavior of the L-scheme over the optimized $L=L_1$, we refrained from including it in Algorithm \ref{alg: a-posteriori} for the sake of simplicity.

\section{Computation of equilibrated flux}\label{sec: eq flux}
Recalling \Cref{def:equilibrated flux,def:equilibrated flux Newton}, let us propose a simple algorithm to compute an equilibrated flux $\bm{\sigma}_h \in \bm{{\rm RT}}_p(\calT_h)\,\cap\, \bm{H}({\rm div},\Om)$ satisfying $\del\cdot \bm{\sigma}_h=\Pi_h f$ in $\calT_{\rm deg}^{i,\eps}$, and $\del\cdot \bm{\sigma}_h=0$ otherwise, where $f\in L^2(\Om)$. Defining $\bm{Q}_h:=\bm{{\rm RT}}_p(\calT_h)\,\cap\, \bm{H}({\rm div},\Om)$ and $\Tilde{V}_h:= \{v_h\in \calP_p(\calT_h)|\; {\rm Tr}_{\p\Om}(v_h)=0 \}$, we seek a pair $(\bm{\sigma}_h, r_h)\in \bm{Q}_h\times \Tilde{V}_h$ that satisfies the mixed finite element problem,
\begin{subequations}
    \begin{align}
     (K(1)^{-1}\bm{\sigma}_h, \bm{q}_h)&= (r_h,\del\cdot \bm{q}_h),\quad &\forall\, \bm{q}_h\in \bm{Q}_h,\\
     (\del \cdot\bm{\sigma}_h, v_h)&= (f, v_h),\quad &\forall\, v_h\in \Tilde{V}_h.
\end{align}
\end{subequations}
The advantage of this flux is that it minimizes $\|K(1)^{-\frac{1}{2}}\bm{\sigma}_h\|$ which appears in the estimates in \Cref{prop:LtoN,prop:NtoN}. For practical purposes, a much coarser mesh can be used outside of $\calT_{\rm deg}^{i,\eps}$ to compute it, and the stiffness matrix can be precomputed to accelerate the computation.

\end{appendices}

\section*{Acknowledgements}
The work of JWB is funded in part through the Center of Sustainable Subsurface Resources (Norwegian Research Council project 331841) and the `FracFlow' project funded by Equinor, Norway through Akademiaavtalen. KM acknowledges the support of FWO (Fonds Wetenschappelijk Onderzoek) for funding him through the `Junior Postdoctoral Fellowship' and to Akademiaavtalen for funding his visit to the University of Bergen.

\bibliographystyle{elsarticle-num}





\end{document}